\RequirePackage{fix-cm}
\documentclass[a4paper,11pt]{article}
\usepackage[a4paper, total={6in, 8in}]{geometry}
\usepackage[square,numbers]{natbib}
\usepackage{graphicx}
\usepackage[utf8]{inputenc}
\usepackage{amsthm, amsmath, amsfonts, amssymb}
\let\*\relax
\usepackage{thmtools}

\usepackage[mathscr]{euscript}
\usepackage{xspace}
\usepackage{color, xcolor}
\usepackage{enumerate}
\usepackage[shortlabels]{enumitem}
\usepackage[english]{babel}
\usepackage{caption}
\usepackage{verbatim}
\usepackage{algorithm}
\usepackage[noend]{algpseudocode}

\usepackage{eqparbox}
\algnewcommand\LeftComment[1]{$\triangleright$ \eqparbox{COMMENT}{#1} \hfill }

\usepackage{pgf, tikz}
\usetikzlibrary{arrows, automata}
\usetikzlibrary{fit}
\usetikzlibrary{shapes.geometric}
\usepackage{ifthen}
\usetikzlibrary{calc}
\tikzset{
  prefix after node/.style={prefix after command=\pgfextra{#1}},
  /semifill/ang/.initial=45,
  /semifill/upper/.initial=none,
  /semifill/lower/.initial=none,
  semifill/.style={
    circle, draw,
    prefix after node={
      \pgfqkeys{/semifill}{#1}
      \path let \p1 = (\tikzlastnode.north), \p2 = (\tikzlastnode.center),
                \n1 = {\y1-\y2} in [radius=\n1]
            (\tikzlastnode.\pgfkeysvalueof{/semifill/ang}) 
            edge[
              draw=none,
              fill=\pgfkeysvalueof{/semifill/upper},
              to path={
                arc[start angle=\pgfkeysvalueof{/semifill/ang}, delta angle=180]
                -- cycle}] ()
            (\tikzlastnode.\pgfkeysvalueof{/semifill/ang}) 
            edge[
              draw=none,
              fill=\pgfkeysvalueof{/semifill/lower},
              to path={
                arc[start angle=\pgfkeysvalueof{/semifill/ang}, delta angle=-180]
                -- cycle}] ();}}}

\usepackage{subcaption}
\usepackage{authblk}
\declaretheorem[numbered=no, name=Maximum Weight Subgraph Problem]{mws-problem}

\newcommand{\poly}{\mathcal{P}}
\DeclareMathOperator{\conv}{conv}
\DeclareMathOperator*{\val}{val}
\newcommand{\R}{\mathbb{R}}
\newcommand{\B}{\mathbb{B}}
\newcommand{\Z}{\mathbb{Z}}

\newcommand{\Q}{\mathbb{Q}}

\newcommand{\NP}{\textsf{NP}}
\newcommand{\Pclass}{\textsf{P}}
\newcommand{\bigO}{\mathcal{O}}

\newcommand{\subpart}{\mathcal{V}}

\usepackage{pgfplotstable}
\usepackage{pgfplots}
\usetikzlibrary{pgfplots.groupplots}
\usepgfplotslibrary{colorbrewer}
\pgfplotsset{compat = newest} 
\pgfplotsset{cycle list/Set1-8} 
\usetikzlibrary{pgfplots.statistics, pgfplots.colorbrewer}

\usepackage{perf}

\newtheorem{theorem}{Theorem}
\newtheorem{proposition}{Proposition}
\newtheorem{lemma}{Lemma}
\newtheorem{corollary}{Corollary}

\providecommand{\keywords}[1]{\textbf{\textit{Keywords:}} #1}

\begin{document}

\title{On the connected (sub)partition polytope\thanks{This research was supported by Internal Funds KU Leuven (Grant C14/22/026).}
}

\author{Phablo F. S. Moura}
\author{Hande Yaman}
\author{Roel Leus}

\affil{Research Center for Operations Research \& Statistics, KU Leuven, Leuven, Belgium \\
\texttt{\{phablo.moura, hande.yaman, roel.leus\}@kuleuven.be}}

\maketitle

\begin{abstract}
Let $k$ be a positive integer and let $G$ be a graph with $n$ vertices.
%We denote by~$V(G)$ and~$E(G)$ the sets of vertices and edges of $G$, respectively. For a positive integer $\ell$, we let $[\ell]=\{1, \ldots, \ell\}$.
A connected $k$-subpartition of $G$ is a collection of $k$ pairwise disjoint sets (a.k.a. classes) of vertices in $G$ such that each set induces a connected subgraph.
The connected $k$-subpartition polytope of $G$, denoted by $\poly(G,k)$, is defined as the convex hull of the incidence vectors of all connected $k$-subpartitions of $G$.
Many applications arising in off-shore oil-drilling, forest planning, image processing, cluster analysis, political districting, police patrolling, and biology are modeled in terms of finding connected (sub)partitions of a graph.
This study focuses on the facial structure of~$\poly(G,k)$ and the computational complexity of the corresponding separation problems.
We first propose a set of valid inequalities having non-zero coefficients associated with a single class that extends and generalizes the ones in the literature of related problems, show sufficient conditions for these inequalities to be facet-defining, and design a polynomial-time separation algorithm for them.
We also devise two sets of inequalities that consider multiple classes, prove when they define facets, and study the computational complexity of associated separation problems. 
Finally, we report on computational experiments showing  the usefulness of the proposed inequalities.
% \subclass{MSC 90C57 \and MSC 68W40}
\end{abstract}


\section{Introduction}

Let $k$ be a positive integer and let $G$ be a graph with $n$ vertices.
We denote by~$V(G)$ and~$E(G)$ the sets of vertices and edges of $G$, respectively. For a positive integer $\ell$, we let $[\ell]=\{1, \ldots, \ell\}$.
A \emph{connected $k$-subpartition} $\subpart=\{V_1, \ldots, V_k\}$ of $G$ is a collection of $k$ pairwise disjoint sets of vertices in $V(G)$ such that each set induces a connected subgraph, that is, $V_i \cap V_j = \emptyset$ for all $i,j \in [k]$ with $i\neq j$ and the subgraph of $G$ induced by $V_i$, denoted by $G[V_i]$, is connected for every $ i  \in [k]$.
A collection $\subpart$ is a \emph{connected $k$-partition} if additionally $\bigcup_{i \in [k]} V_i = V(G)$.

Given a set of vertices $V'\subseteq V(G)$ and $c \in [k]$, we define the vector $e(V', c) \in \{0,1\}^{kn}$ such that its non-zero entries are precisely $e(V',c)_{v,c}= 1$ for all $v \in V'$. 
For the sake of simplicity, we write $e(H,c)$ rather than  $e(V(H), c)$ when $H$ is a graph.
Let $\subpart=\{V_1, \ldots, V_k\}$ be a $k$-subpartition  of $G$.
The non-zero entries of its \emph{incidence vector} $\chi(\subpart)$  are precisely $\chi(\subpart)_{v,c} =1$ for every $c \in [k]$, $v \in V_c$, that is, $\chi(\subpart) = \sum_{c\in [k]} e(V_c,c)$.
The connected $k$-subpartition polytope of $G$, denoted by $\poly(G,k)$, is defined as
\[\poly(G,k) = \conv \{\chi(\subpart) \in \{0,1\}^{nk} \colon \subpart \text{ is a connected $k$-subpartition of $G$}\}.\]

Many combinatorial problems deal with finding (sub)partitions of the set of vertices of a graph into classes that induce connected subgraphs, or equivalently integer points in the polytope~$\poly$.
A well-studied example is the \textsc{Maximum Weight Connected Subgraph} problem, which receives as input a vertex-weighted graph $G$ (weights may be negative), and consists in finding a connected subgraph of $G$ that maximizes the sum of the weights of its vertices 
(see {\'Alvarez-Miranda et al.~\cite{Alvares-Miranda2013}}, and Hochbaum and Pathria~\cite{hochbaum1994node}).

\textsc{Convex Recoloring} is another problem involving partitioning graphs into connected subgraphs that has received considerable attention in the last years (see Camp\^elo et al.~\cite{CAMPELO202254}, and Chopra et al.~\cite{chopra2017extended}).
This problem was introduced by Moran and Snir~\cite{MorSni08}, and consists in finding a (re)coloring of a given colored graph such that each color class in the recoloring induces a connected subgraph, and the number of recolored vertices is minimized.
Note that the concept of coloring here differs from the classical definition of (proper) coloring, where pairs of vertices linked by an edge are assigned distinct colors.

In the \textsc{Balanced Connected Partition} problem, the goal is to partition a graph into a given number of connected subgraphs so that the size of every subgraph is roughly the same.
This concept of balance can be modeled, for example, as minimizing (resp.\ maximizing) the size of the largest (resp.\ smallest) subgraph in the partition, or  minimizing the maximum difference of sizes between two subgraphs.
These objectives yield equivalent problems when one has to partition into two classes, but lead to distinct problems when the desired partition has more than two subgraphs 
(see Lucertini et al.~\cite{LucPerSim93}).
A related problem investigated by Lovász~\cite{Lovasz77} and Gy{\"o}ri~\cite{Gyori78} consists in partitioning a connected graph into a given number of connected subgraphs with prescribed orders.

Another example of connected partition problem is \textsc{Connected Maximum $k$-Cut},  a variant of the classical \textsc{Maximum $k$-Cut} in which the goal is to find a partition of the vertices into $k$ classes inducing connected subgraphs such that the number of edges linking vertices in distinct classes is maximized.
Recently, Hojny et al.~\cite{Hojny21} proposed mixed-integer linear programming formulations for an edge-weighted version of this problem, and reported on extensive computational experiments.

The idea of dividing graphs into connected parts (classes) plays an essential role in many applications arising in off-shore oil-drilling, forest planning, image processing, cluster analysis, political districting, police patrolling, and biology.
For more details on these applications, we refer to~\cite{ArrMarPenRic21,Carvajal13,hochbaum1994node,LucPerSim93,MarSimNal97,MIYAZAWA2021826,MorSni08}.

In this work, we study the facial structure of the connected $k$-subpartition polytope, which is naturally modeled as follows.
Given any pair of non-adjacent vertices $\{u,v\} \notin E(G)$, a $u,v$-\emph{separator} is a set of vertices $Z\subseteq V(G)\setminus \{u,v\}$ such that $u$ and $v$ belong to different components of $G-Z=G[V(G)\setminus Z]$. 
Let us denote by~$\Gamma(u,v)$ the set of all minimal $u,v$-separators in $G$.
Consider now the following inequalities on variables~$x \in \R^{nk}$.
\begin{align}  
    \:& \sum_{c \in [k]} x_{v,c}  \leq 1 &  \!\!\! \!\!\! \forall v \in V,\label{ineq:cover} \\ 
    \:& x_{u,c} +  x_{v,c} - \sum_{z \in Z} x_{z,c}  \leq 1 &  \!\!\! \!\!\! \forall  
    \{u,v\} \notin E(G), Z \in \Gamma(u,v), c \in [k],\label{ineq:cut} \\
    \:& x_{v,c} \ge 0 & \!\!\! \!\!\! \forall v \in V \text{ and } c \in [k].\label{ineq:non-negativity} 
\end{align}

Let us denote by $\poly'(G,k)$ the set of all vectors~$x \in \R^{nk}$ satisfying~\eqref{ineq:cover}-\eqref{ineq:non-negativity}.
Observe that if two non-adjacent vertices $u$ and $v$ belong to a same class~$c$, then there must exist a path between $u$ and $v$ such that every internal vertex also belongs to class $c$, and such a path intersects every $u,v$-separator in $G$. 
It follows from this remark that \(\poly(G,k) = \conv(\poly'(G,k)\cap \{0,1\}^{nk})\).

The facial structure of the polytope $\poly(G,k)$ was studied before by Camp\^elo et al.~\cite{CAMPELO2013} in the context of a polyhedral approach to \textsc{Convex Recoloring}.
Chopra et al.~\cite{chopra2017extended} provided a strong extended formulation for the problem on trees, and reported on computational experiments.
We remark that \textsc{Convex Recoloring} is $\NP$-hard even on paths, as shown by Moran and Snir~\cite{MorSni08}. 
 Indeed, Moura and Wakabayashi~\cite{MOURA2020252} proved that this problem on n-vertex bipartite graphs cannot be approximated within a factor of $n^{1-\varepsilon}$ for any $\varepsilon > 0$, unless $\Pclass=\NP$.

Wang et al.~\cite{wang2017imposing} presented a detailed polyhedral study for the case  $k=1$, which is the so-called \emph{connected subgraph polytope}.
More precisely, assuming that $G$ is connected, they showed that~\eqref{ineq:cut} induces a facet of $\poly(G,1)$ if and only if $Z$ is minimal.
They also proved that the \emph{indegree} inequalities, known to induce all nontrivial facets of $\poly(G,1)$ when $G$ is a tree (see Korte et al.~\cite{korte2012greedoids}), 
are valid for general graphs and can be separated in linear time. 
The indegree inequalities can be generalized to the case of positive $k$ in a trivial way as follows. 
Consider an orientation~$\vec E$ of the edges in $G$, and define, for each $v \in V(G)$, the \emph{ indegree} $d(v)$ as the number of arcs in $\vec E$ with head~$v$.
The indegree inequality corresponding to~$\vec E$  and $c \in [k]$ is
\begin{equation}\label{ineq:indegree}
    \sum_{v \in V(G)} (1-d(v))x_{v,c} \leq 1.
\end{equation}

The validity of this inequality follows from the fact that $\sum_{v \in V(G')} d(v) \geq |E(G')|$, and $|E(G')|\geq |V(G')|-1$ for any connected subgraph~$G'$ of $G$.
Unlike for the case of $k=1$, for $k\ge 2$ the indegree inequalities and nonnegativity inequalities~\eqref{ineq:non-negativity} are not sufficient to describe $\poly(G,k)$ when $G$ is a tree, as \textsc{Convex Recoloring} is \NP-hard on paths. 
This simple observation suggests that the facial structure of $\poly(G,k)$ with $k\geq 2$ is significantly more involved compared to the one of $\poly(G,1)$. This work is devoted to studying the facial structure of $\poly(G,k)$, and the computational complexity of the  separation problems associated with the families of valid inequalities.

Before concluding this section, we review the results on the dimension of $\poly(G,k)$ and the trivial inequalities of the formulation previously described.

\begin{theorem}[Camp\^elo et al.\cite{CAMPELO2013,CAMPELO2016}]\label{thm:polytope:basics}
The following hold for $\poly(G,k)$:
    \begin{enumerate}[(i)]
        \item The polytope is full-dimensional, i.e., $dim(\poly(G,k))= nk$.
        \item For every $v \in V(G)$, inequality~\eqref{ineq:cover} is facet-defining if $k>1$, or if $k=1$ and $G$ is connected. \label{thm:case:connected}
        \item For every $v \in V(G)$ and $c \in [k]$, inequality~\eqref{ineq:non-negativity} is facet-defining.
    \end{enumerate}
\end{theorem}

One may notice that if $k=1$ and $G$ is not connected, then there exist vertices~$u,v \in V(G)$ belonging to different components of $G$ such that $x_u + x_v \leq 1$ is valid for $\poly(G,1)$, and so inequalities~\eqref{ineq:cover} for $u$ and $v$ are dominated.
This explains why one has to suppose that $G$ is connected in Theorem~\ref{thm:polytope:basics}\ref{thm:case:connected}.
The characterization of the facet-defining inequalities~\eqref{ineq:cut} shown by Wang et al.~\cite{wang2017imposing} also assumes that~$G$ is connected.
Hence, even for the basic inequalities in the formulation, their facetness may depend on the graph being connected.
For the sake of simplicity in our proofs, we henceforth assume that $G$ is a connected graph.

\subsection*{Contributions}

In Section~\ref{sec:single} we introduce a class of inequalities for $\poly(G,k)$ that generalizes both the connectivity inequalities~\eqref{ineq:cut} and the indegree inequalities~\eqref{ineq:indegree}.
The proposed class contains all nontrivial inequalities studied by Wang et al.~\cite{wang2017imposing} for the connected subgraph polytope (i.e., $k=1$),  and by Chopra et al.~\cite{chopra2017extended} and Camp\^elo et al.~\cite{CAMPELO2016} for $\poly(G,k)$ with $k\ge 1$.
Furthermore, it contains new strong valid inequalities.
We also prove sufficient (weak) conditions under which the proposed generalized inequalities are facet-defining for $\poly(G,k)$.
If the input graph~$G$ is a tree, we show that  all facets of $\poly(G,k)$ having non-zero coefficients associated with a single class in $[k]$ are induced by indegree inequalities~\eqref{ineq:indegree}.

The separation problem associated with these inequalities is investigated in Section~\ref{sec:separation-single}.
More precisely, we design a polynomial-time  algorithm that reduces the input instance to a quadratic number of bipartite instances of the \textsc{Minimum Weight Vertex Cover} problem, when the partition defining the inequality is fixed. 
In general, this algorithm runs in~$\bigO(k|V(G)||E(G)|)$ time.
When separating inequalities~\eqref{ineq:cut} and~\eqref{ineq:indegree}, this algorithm achieves the same running times as the corresponding separation routines for these inequalities described in the literature.
The proposed routine runs in $\bigO(k(|V(G)|+|E(G)|))$  time when separating a set of inequalities that strictly contains the indegree inequalities~\eqref{ineq:indegree}.

In Section~\ref{sec:multi}, we propose two sets of valid inequalities, namely multiway and pairing inequalities,  that combine multiple classes.
We characterize the facets induced by multiway inequalities, and present sufficient conditions for the pairing inequalities to be facet-defining.
We show in Section~\ref{sec:separation-multi} that the separation problem associated with multiway inequalities is $\NP$-hard (if the stable set which defines it is given) as this separation boils down to solving the \textsc{Multiway Cut} problem.
Moreover, we show that the separation problem of the pairing inequalities can solved in polynomial time if the graph is a tree.
The proposed algorithm computes a minimum-cost flow in a network which is constructed with the vertices of the input graph and additional vertices corresponding to the classes.

In Section~\ref{sec:experiments}, we investigate the computational impact of the proposed single-class and multiclass inequalities in solving instances generated following 
Wang et al.~\cite{wang2017imposing} for the Maximum-Weight Connected Subgraph problem, as well as on computationally more challenging instances.
The computational results indicate that the proposed inequalities lead to shorter running times and smaller gaps than the inequalities described in the literature.

Concluding remarks and possible directions for further research are discussed in Section~\ref{sec:conclusion}.

 \section{Single-class inequalities} \label{sec:single}

In this section, we consider valid (non-trivial) inequalities for $\poly(G,k)$ where all non-zero coefficients are associated with a same class $c \in [k]$.
Let us denote by $(\pi,\pi_0) \in \R^\ell \times \R$ the inequality $\pi x \leq \pi_0$ where~$x \in \R^\ell$. 
An inequality $(\pi,\pi_0) \in \R^{nk} \times \R$ is said to be a \emph{single class} inequality for $\poly(G,k)$ if $(\pi,\pi_0)$ is a valid inequality for $\poly(G,k)$, and there exists~$c \in [k]$ such that $\pi_{v,b}=0$ for all $v \in V(G)$ and $b \in [k]\setminus\{c\}$.
We say that an inequality of $\poly(G,k)$ is nontrivial if it is different from $(0, 0)$ and from the nonnegativity constraints~\eqref{ineq:non-negativity}.

\subsection{Facets from $\poly(G,1)$ to $\poly(G,k)$}

First we show that every facet-defining inequality  of $\poly(G,1)$ can be naturally lifted to a single-class inequality of $\poly(G,k)$ that is also facet defining.
In particular, we shall see that the connectivity inequalities~\eqref{ineq:cut}  and the indegree inequalities~\eqref{ineq:indegree} induce facets of $\poly(G,k)$ under the same hypothesis shown by Wang et al.~\cite{wang2017imposing} for $\poly(G,1)$.

\begin{lemma}\label{lemma:lifting}
    Let $(\pi, \pi_0) \in \R^n \times \R$ be a nontrivial valid inequality of $\poly(G,1)$ 
     different from $x_{v}\leq 1$ for all $v\in V(G)$.
Let~$k \ge 2$ be an integer, let~$c \in [k]$, and let $(\pi^c, \pi_0 ) \in \R^{nk} \times \R$ be a valid inequality of $\poly(G,k)$, where the (possibly) non-zero entries of $\pi^c$ are precisely $\pi^c_{v,c} = \pi_v$ for all $v \in V(G)$.
    It holds that~$(\pi,\pi_0)$ is facet defining for $\poly(G,1)$ if and only if $(\pi^c, \pi_0 )$ is facet defining for~$\poly(G,k)$.
\end{lemma}
\begin{proof}
    Suppose first that $(\pi,\pi_0)$ defines a facet $F$ of $\poly(G,1)$.
    Let $\tilde F$ be the face of $\poly(G,k)$ induced by $(\pi^c, \pi_0 )$, that is, $\tilde F = \{x \in \poly(G,k) : \pi^c x = \pi_0\}$.
    By Theorem~\ref{thm:polytope:basics}, we have $\dim(F)=n-1$, and so $F$ contains a set~$X = \{x^j\}_{j \in [n]}$ of affinely independent vectors such that $x^j \in \{0,1\}^n$ for all $j \in [n]$.
    Observe now that, for each $v \in V(G)$, there is a vector $r^v \in X$ such that $r^v_{v}=0$, otherwise $F \subset \{x \in \poly(G,1) : x_v=1\}$, a contradiction to the assumption that $F$ is a facet of $\poly(G,1)$.

    In what follows, we  show how to lift the vectors in $F$ to the space $\{0,1\}^{nk}$.
    For each $j \in [n]$, we define a vector $\tilde x^j \in \{0,1\}^{nk}$ such that, for every~$v \in V(G)$, $\tilde x^j_{v,c} = x^{j}_v$, and $\tilde  x^{j}_{v,b}=0$ for all $b \in [k]\setminus \{c\}$.
    Similarly, for each $v \in V(G)$, let $\tilde r^v=\tilde x^j$ where $r^v=x^j$ for some $j \in [n]$.
    Note  that every vector described above corresponds to a connected $k$-subpartition of $G$ where class $c$ is the single non-empty class, and that it belongs to $\tilde F$.

    Let us define $R = \{\tilde r^v+ e(v,b) : v \in V(G) \text{ and } b \in [k]\setminus \{c\}\}$.
    By the definition of $r^v$, $v$ is not assigned to any class in the corresponding connected $k$-subpartition, and thus $r^v+e(v,b)$ belongs to $\tilde F$ for all $b \in [k]\setminus \{c\}$.
    Let $\tilde X = R\cup \{\tilde x^j\}_{j \in [n]}$, and note that $|\tilde X| = n(k-1)+ n=nk$.
    One may easily verify that $\tilde X \subseteq \tilde F$.
    Since $X$ is an affinely independent set and by the construction of the vectors, it holds that $\tilde X$ is also  affinely independent.
    As a consequence, $\dim(\tilde F) = nk-1$ and so $\tilde F$ is a facet of $\poly(G,k)$.

    Suppose now that $(\pi^c, \pi_0)$ induces a facet $\tilde F$ of $\poly(G,k)$.
    Thus, for every~$v \in V(G)$,  $\tilde F \neq \{ x \in \poly(G,k) :  x_{v,c}=1\}$ since the valid inequality $x_{v,c} \le 1$ is dominated by~\eqref{ineq:cover} if $k\ge 2$.
    Again by Theorem~\ref{thm:polytope:basics}, there are $nk$ affinely independent vectors in $\tilde F$.
    Let us denote by $M$ be the matrix $nk \times nk$ formed by such vectors.
    Consider the submatrix of $M'$ of $M$ formed by the $n$ rows indexed by $c$, and observe that ea ch column of $M'$ belongs to $F$.
    As $M$ contains $nk$ affinely independent rows, the affine rank of $M'$ is $n$.
    Hence there exist $n$ affinely independent vectors in $F$.
    Therefore, $\dim(F)=n-1$ and so $F$ is a facet of $\poly(G,1)$.
 \end{proof}

We next use Lemma~\ref{lemma:lifting} to characterize inequalities~\eqref{ineq:cut} and~\eqref{ineq:indegree} that are facet-defining for $\poly(G,k)$.

\begin{corollary}\label{cor:simple-facets}
    Let $k\geq 2$ be an integer.
    The following hold.
    \begin{enumerate}[(i)]
        \item Inequality~\eqref{ineq:cut} is facet defining for $\poly(G,k)$ if and only if the separator is minimal. \label{cor:item:connectivity}
\item Inequality~\eqref{ineq:indegree} is facet defining for $\poly(G,k)$ if and only if 
        \begin{enumerate}[(a)]
            \item for every $u,v \in V(G)$, there exists at most one directed $u,v$-walk in the corresponding orientation of~$G$; \label{cor:subitem:walks} and
\item the orientation of~$G$ contains at least two vertices of indegree zero.\label{cor:subitem:zero}
        \end{enumerate}
        \label{cor:item:indegree}
    \end{enumerate}
\end{corollary}
\begin{proof}
    Wang et al.~\cite{wang2017imposing} proved that~\eqref{ineq:cut}, for non-adjacent vertices~$u,v \in V(G)$, and $(u,v)$-separator $Z$,  defines a facet  of  $\poly(G,1)$  if and only if~$Z$ is  minimal. As $G$ is connected, this inequality is different from $x_v\leq1$ for all $v \in V(G)$. Hence~\ref{cor:item:connectivity} holds using   Lemma~\ref{lemma:lifting}.
    
    To prove~\ref{cor:item:indegree}, suppose first that~\eqref{ineq:indegree} induces a facet~$F$ of $\poly(G,k)$, and let $D$ denote the orientation of $G$ that defines such inequality.
    By Lemma~\ref{lemma:lifting}, the corresponding facet~$F'$ of $\poly(G,1)$ is different from $\{x \in \poly(G,1) : x_v=1\}$ for all $v \in V(G)$.
    Wang et al.~\cite{wang2017imposing} proved that~$F'$ is a facet of  $\poly(G,1)$ if and only if~(a) holds.
    Due to (a) and $F'\neq \{x \in \poly(G,1) : x_v=1\}$ for all $v \in V(G)$, $D$ contains at least two vertices of indegree zero.

    To show the converse, consider an orientation $D$ of $G$ satisfying (a) and~(b). 
    Let~$F'$ be the face of $\poly(G,1)$ induced by the inequality~\eqref{ineq:indegree} which is associated with $D$.
    Since $D$ has at least two vertices with indegree zero, the inequality inducing $F'$ contains at least two coefficients equal to one, and so it is different from $x_v\leq1$ for all $v \in V(G)$.
    Since~(a) holds, $F'$ is a facet of~$\poly(G,1)$ by Wang et al.
    Therefore, the claimed result follows from Lemma~\ref{lemma:lifting}.
 \end{proof}

The facial structure of the connected $k$-subpartition polytope on trees was studied before in the context of the \textsc{Convex Recoloring} problem to model applications in biology involving phylogenetic trees (see Camp\^elo et al.~\cite{CAMPELO2016}, and  Chopra et al.~\cite{chopra2017extended}). 
In what follows, we use Lemma~\ref{lemma:lifting} to characterize the facets of~$\poly(G,k)$ where $G$ is a tree.

Korte et al.~\cite{korte2012greedoids}
(Theorem 3.6, p.~168) showed an explicit description of $\poly(G,1)$  in terms of the indegree inequalities when $G$ is a tree.

\begin{theorem}[Korte et al.~\cite{korte2012greedoids} (see~\cite{wang2017imposing})]\label{thm:tree-polytope}
    Let $G$ be a tree.
    It holds that 
    \[\poly(G,1) = \{x \in \R^n_\ge : x \text{ satisfies all indegree inequalities }\eqref{ineq:indegree}\}.\]
Moreover, each inequality~\eqref{ineq:indegree} induces a facet of $\poly(G,1)$.
\end{theorem}

\begin{theorem}
    Let $G$ be a tree, let $k\ge2$ be an integer, and let $(\pi,\pi_0) \in \R^{nk} \times \R$ be a  nontrivial  single-class inequality of $\poly(G,k)$.
    Then, $(\pi,\pi_0)$ induces a facet of $\poly(G,k)$ if and only if $(\pi,\pi_0)$ is equal to~\eqref{ineq:indegree} defined by an orientation of $G$ with at least two vertices of indegree zero.
\end{theorem}
\begin{proof}
    Let $F$ be the facet of $\poly(G,k)$ induced by $(\pi,\pi_0)$.
    Applying Lemma~\ref{lemma:lifting} on~$F$ yields a facet $F'$ of $\poly(G,1)$ that is different from $\{x \in \R^n : x_v=1\}$ for all $v \in V(G)$.
    By Theorem~\ref{thm:tree-polytope}, $F'$ is induced by an indegree inequality~\eqref{ineq:indegree}, and so $(\pi,\pi_0)$ is equal to \eqref{ineq:indegree} for some $c \in [k]$. 
    Additionally, the orientation of $G$ defining such inequality contains at least two vertices of degree zero since $G$ has no cycle, and $F' \neq \{x \in \R^n : x_v=1\}$ for all $v \in V(G)$.
    
    Consider  an orientation of $G$ having at least two vertices of indegree zero. 
    This orientation defines an inequality~\eqref{ineq:indegree} of $\poly(G,1)$ containing at least two  coefficients equal to one, and  so it is different  from $x_v \leq 1$ for all $v \in V(G)$.
    Moreover, such inequality induces a facet of $\poly(G,1)$ by Theorem~\ref{thm:tree-polytope}.
    The result follows from Lemma~\ref{lemma:lifting}. 
 \end{proof}

We remark that all (nontrivial) facets of $\poly(G,k)$ shown by Chopra et al.~\cite{chopra2017extended}, and by Camp\^elo et al.~\cite{CAMPELO2016} (when $G$ is a tree) are induced by single-class inequalities.
As a consequence of the previous theorem, we conclude that all such facets are induced by indegree inequalities~\eqref{ineq:indegree}.

For general graphs $G$, Camp\^elo et al.~\cite{CAMPELO2016} devised single-class inequalities for~$\poly(G,k)$ generalizing~\eqref{ineq:cut}, and showed sufficient conditions for  facetness.
These inequalities are a particular case of the generalized connectivity inequalities introduced next.

\subsection{Generalized connectivity inequalities}
\label{sec:general-single}

Before showing a class of inequalities that generalizes both~\eqref{ineq:cut} and~\eqref{ineq:indegree}, we introduce some additional notation and definitions.
Let $\ell \in [n]$.
Given a set $S \subseteq V(G)$ with $|S|=\ell$, and an $\ell$-partition $\mathcal{W} = \{W_1, \ldots, W_{\ell}\}$ of $V(G)$, we say that $\mathcal{W}$ \emph{agrees} with $S$ if each vertex of $S$ belongs to exactly one class in $\mathcal{W}$.
We denote by $E(\mathcal{W})$ the set of all edges of~$G$ with endpoints in different classes of $\mathcal{W}$, that is, $E(\mathcal{W}) = \{\{u,v\} \in E(G) \colon u \in W_i, v\in W_j \text{ for some } i,j \in [\ell] \text{ with } i\neq j \}$.
Let $\vec E(\mathcal{W})$ be the set of arcs obtained in an arbitrary (but fixed) orientation of the edges in $E(\mathcal{W})$, that is, $\vec E(\mathcal{W})= \{(u,v)  \in V(G)\times V(G) : \{u,v\} \in E(\mathcal{W}) \}$ is asymmetric.

We now define a function $\hat d \colon V(G) \to [\ell]$ such that $\hat d(v) = |\{j \in [\ell] : (u,v) \in \vec E(\mathcal{W}) \text{ and } u \in W_j\}|$.
Intuitively, $\hat d(v)$ corresponds to the number of different classes of $\mathcal{W}$ that contain a vertex that is the tail of some arc with head $v$.
Clearly, every vertex~$v$ that is not an endpoint of an edge in $E(\mathcal{W})$ satisfies
 $\hat d(v)=0$.

\begin{proposition}\label{prop:validity:general-cut}
    Let $c \in [k]$, $S \subseteq V(G)$,  $\mathcal{W}$ be a $|S|$-partition of~$V(G)$ that agrees with $S$, and $\vec E(\mathcal{W})$ be an orientation of $E(\mathcal{W})$ (with $\hat d $ defined accordingly).
The generalized connectivity inequality 
    \begin{equation}\label{ineq:cut-general}
\sum_{u \in S}(1-\hat d(u))x_{u,c} - \sum_{v \in V(G)\setminus S} \hat d(v) x_{v,c} \leq 1
    \end{equation}
    is valid for $\poly(G,k)$.
\end{proposition}
\begin{proof}
    Let $H$ be any connected subgraph of $G$. 
    Suppose that there are $p \leq \ell$ classes of $\mathcal{W}$ that are intersected by $H$, that is, $p = |\{W \in \mathcal{W} \colon V(H)\cap W \neq \emptyset\}|$.
    Clearly, it holds that $|S\cap V(H)| \leq p$.
Since $H$ is connected, it follows from the definition of $\hat d$ that
    \(\sum_{v \in V(H)} \hat d(v) \geq p-1.\)
    Hence, we have 
    \begin{equation}\label{eq:validity}
        |S\cap V(H)| - \sum_{v \in V(H)} \hat d(v) \le  1.
    \end{equation}

    Consider now a vector $\bar x \in \poly(G,k) \cap \{0,1\}^{nk}$.
    Note that the vertex set $\{v \in V(G) :  \bar x_{v,c}=1\}$ induces a connected subgraph, say~$H_c$, of $G$.
    Since $\sum_{u \in S}\bar x_{u,c} = |S\cap V(H_c)|$ and $\sum_{v \in V(H_c)} \hat d(v) = \sum_{v \in V(G)} \hat d(v) \bar x_{v,c}$,  inequality~\eqref{eq:validity} implies that $\bar x$ satisfies~\eqref{ineq:cut-general} for $c$.
    Therefore, inequality~\eqref{ineq:cut-general} is valid for $\poly(G,k)$.
 \end{proof}

We next show an example of a fractional solution in $\poly'(G,k)$ and an inequality~\eqref{ineq:cut-general} that cuts it off.
Consider a path $P$ on $v_1\ldots v_5$, fix $c \in [k]$, and define the vector $\bar x \in \R^{5k}$  such that its non-zero entries are $\bar x_{v_1,c}=\bar x_{v_3,c}=\bar x_{v_5,c}=1/2$.
This example is depicted in Figure~\ref{fig:example:connectivity}.
It is clear that $\bar x \in \poly'(P,k)$.
However, $\bar x$ violates inequality~\eqref{ineq:cut-general} for $S=\{v_1, v_3, v_5\}$, $\mathcal{W} = \{\{v_1,v_2\}, \{v_3\}, \{v_4,v_5\}\}$, and orientation $\{(v_3,v_2),(v_3,v_4)\}$, that is, $x_{v_1,c}+x_{v_3,c}+x_{v_5,c}-x_{v_2,c}-x_{v_4,c}\leq 1$. 

\begin{figure}[tbh!]
\centering
\begin{subfigure}{.47\textwidth}
    \centering
    \begin{tikzpicture}[
auto,
            node distance = 1.5cm, semithick, scale=0.8
]

        \tikzstyle{every state}=[
            draw = black,
            semithick,
            fill = white,
            minimum size = 4mm,
            scale=0.9,
            inner sep = 1pt
        ]

        \node[state] (v1) {$v_1$};
        \node[state] (v2) [right of=v1] {$v_2$};
\node[state] (v3) [right of=v2] {$v_3$};
\node[state] (v4) [right of=v3] {$v_4$};
        \node[state] (v5) [right of=v4] {$v_5$};

        \node [above=0.3cm, scale=0.8] at (v1) {$1/2$};
        \node [above=0.4cm, scale=0.8] at (v2) {$0$};
        \node [above=0.3cm, scale=0.8] at (v3) {$1/2$};
        \node [above=0.4cm, scale=0.8] at (v4) {$0$};
        \node [above=0.3cm, scale=0.8] at (v5) {$1/2$};
        
        \path (v1) edge node{} (v2);
        \path (v2) edge node{} (v3);
        \path (v3) edge node{} (v4);
        \path (v4) edge node{} (v5);
    \end{tikzpicture}\caption{} 
\end{subfigure}\hfill
\begin{subfigure}{.47\textwidth}
    \centering
    \begin{tikzpicture}[
auto,
            node distance = 1.5cm, semithick, scale=0.8
]
\tikzstyle{every state}=[
            draw = black,
            semithick,
            fill = white,
            minimum size = 4mm,
            scale=0.9,
            inner sep = 1pt
        ]

        \node[state, very thick] (v1) {$v_1$};
        \node[state] (v2) [right of=v1] {$v_2$};
        \node[state, very thick] (v3) [right of=v2] {$v_3$};
\node[state] (v4) [right of=v3] {$v_4$};
        \node[state, very thick] (v5) [right of=v4] {$v_5$};

        \node[draw,dotted,fit=(v1) (v2)] {};
        \node[draw,dotted,fit=(v3)] {};
        \node[draw,dotted,fit=(v4) (v5)] {};
        
        \path[-] (v1) edge node{} (v2);
        \path[<-] (v2) edge node{} (v3);
        \path[->] (v3) edge node{} (v4);
        \path[-] (v4) edge node{} (v5);
    \end{tikzpicture}
    \caption{}
\end{subfigure}\hfill
\caption{Example of a fractional solution that violates a general connectivity inequality. (a) shows a graph~$G$ with vertex weights given by a fractional solution in $\poly'(G,k)$ for a fixed $c \in [k]$.
    (b) depicts an orientation $\vec E$ of $E(\mathcal{W})$, where the dotted rectangles represent the partition $\mathcal{W}=\{\{v_1,v_2\}, \{v_3\}, \{v_4,v_5\}\}$ and thicker circles the vertices in $S=\{v_1,v_3,v_5\}$.}
    \label{fig:example:connectivity}
\end{figure}

As we shall see,  connectivity inequalities and  indegree inequalities are special cases of the generalized connectivity inequalities.  
\begin{proposition}\label{prop:connectivity}
Connectivity inequalities~\eqref{ineq:cut} for minimal separators are a particular case of~\eqref{ineq:cut-general} with~$S=\{u,v\}$.
\end{proposition}
\begin{proof}

Consider $u,v \in V(G)$ non-adjacent vertices, $Z \in \Gamma(u,v)$, and $c \in [k]$. 
Let~$U$ be the vertices of the connected component of $G-Z$ containing $u$.
Observe that the partition $\mathcal{W}=\{U,\bar U\}$ agrees with the set $\{u,v\}$, where $\bar U=V(G)\setminus U$.
Suppose  now that all arcs are oriented from $U$ to $\bar U$, and note that~$\hat d(v)=1$ if $v \in Z$, and $\hat d(v)=0$ otherwise. 
 \end{proof}

\begin{proposition}\label{prop:indegree}
Indegree  inequalities~\eqref{ineq:indegree} are precisely~\eqref{ineq:cut-general}  with $S=V(G)$.
\end{proposition}
\begin{proof}
The class of inequalities~\eqref{ineq:cut-general} with $S=V(G)$ and $\mathcal{W} = \{\{v\}\}_{v \in V(G)}$ are precisely inequalities~\eqref{ineq:indegree}. 
 \end{proof}

In Figure~\ref{fig:example:indegree}, we show a graph and a fractional vector that  satisfies all connectivity and all indegree inequalities (i.e.,~\eqref{ineq:cut} and~\eqref{ineq:indegree}), but violates a certain inequality~\eqref{ineq:cut-general} with a partition~$\mathcal{W}$ of size $3$.
Figure~\ref{subfig:fractional} depicts a graph $G$ with weights given by a vector $\bar x \in \poly'(G,k)$ (for a fixed $c \in [k])$, that is, $\bar x_{v_1,c}=\bar x_{v_4,c}=\bar x_{v_7,c}=1/2$, $\bar x_{v_3,c}=\bar x_{v_6,c}=1/8$, and  $\bar x_{v_2,c}=\bar x_{v_5,c}=1/16$.
One may verify that, for any orientation of $E(G)$, $\bar x$ satisfies inequality~\eqref{ineq:indegree}. 
An orientation that maximizes the left-hand side of~\eqref{ineq:indegree} for $\bar x$ (and $c \in [k]$) is shown in Figure~\ref{subfig:orientation}.
The dotted rectangles in this figure represent the three classes of the partition $\mathcal{W} = \{\{v_1,v_2,v_3\}, \{v_4\}, \{v_5,v_6,v_7\}\}$.
Considering the orientation restricted to $E(\mathcal{W})$, it holds that $\hat d(v_2)=\hat d(v_3)=\hat d(v_5)=\hat d(v_6)=1$ whilst each other vertex $v$ of $G$ has $\hat d(v) =0$.
Therefore, inequality~\eqref{ineq:cut-general} corresponding to $S=\{v_1,v_4,v_7\}$, $\mathcal{W},$ and $ \hat d$ is 
$x_{v_1,c}+x_{v_4,c}+x_{v_7,c} - x_{v_2,c} - x_{v_3,c} - x_{v_5,c} - x_{v_6,c} \le 1$, which is violated by $\bar x$.
As we shall demonstrate in Theorem~\ref{thm:gen-connectiviy:facetness}, this inequality is facet-defining.

\begin{figure}[t!]
\centering
\begin{subfigure}{.47\linewidth}
    \centering
    \begin{tikzpicture}[
auto,
            node distance = 1.5cm, semithick, scale=0.8, 
]

        \tikzstyle{every state}=[
            draw = black,
            semithick,
            fill = white,
            minimum size = 4mm,
            scale=0.9,
            inner sep=1pt
        ]

        \node[state] (v1) {$v_1$};
        \node[state] (v2) [above right of=v1] {$v_2$};
        \node[state] (v3) [below right of=v1] {$v_3$};
\node[state] (v4) [below right of=v2] {$v_4$};
        \node[state] (v5) [above right of=v4] {$v_5$};
        \node[state] (v6) [below right of=v4] {$v_6$};
        \node[state] (v7) [below right of=v5] {$v_7$};

        \node [below=0.3cm, scale=0.8] at (v1) {$1/2$};
        \node [below=0.3cm, scale=0.8] at (v4) {$1/2$};
        \node [below=0.3cm, scale=0.8] at (v7) {$1/2$};

        \node [right=0.3cm, scale=0.8] at (v2) {$1/16$};
        \node [right=0.3cm, scale=0.8] at (v5) {$1/16$};
        \node [right=0.3cm, scale=0.8] at (v3) {$1/8$};
        \node [right=0.3cm, scale=0.8] at (v6) {$1/8$};

        \path (v1) edge node{} (v2);
        \path (v1) edge node{} (v3);
        \path (v3) edge node{} (v2);
        \path (v4) edge node{} (v2);
        \path (v4) edge node{} (v3);
        
        \path (v7) edge node{} (v5);
        \path (v7) edge node{} (v6);
        \path (v6) edge node{} (v5);
        \path (v4) edge node{} (v5);
        \path (v4) edge node{} (v6);

\end{tikzpicture}\caption{} 
    \label{subfig:fractional}
\end{subfigure}\hfill
\begin{subfigure}{.47\linewidth}
    \centering
    \begin{tikzpicture}[
auto,
            shorten > = 1pt, node distance = 1.5cm, semithick, scale=0.8, 
]
\tikzstyle{every state}=[
            draw = black,
            semithick,
            fill = white,
            minimum size = 4mm,
            scale=0.9,
            inner sep =1pt
        ]

        \node[state,very thick] (v1) {$v_1$};
        \node[state] (v2) [above right of=v1] {$v_2$};
        \node[state] (v3) [below right of=v1] {$v_3$};
\node[state, very thick] (v4) [below right of=v2] {$v_4$};
        \node[state] (v5) [above right of=v4] {$v_5$};
        \node[state] (v6) [below right of=v4] {$v_6$};
        \node[state, very thick] (v7) [below right of=v5] {$v_7$};

        \node[draw,dotted,fit=(v1) (v2) (v3)] {};
        \node[draw,dotted,fit=(v5) (v6) (v7)] {};
        \node[draw,dotted,fit=(v4)] {};

        \path[->] (v1) edge node{} (v2);
        \path[->] (v1) edge node{} (v3);
        \path[->] (v3) edge node{} (v2);
        \path[->] (v4) edge node{} (v2);
        \path[->] (v4) edge node{} (v3);
        
        \path[->] (v7) edge node{} (v5);
        \path[->] (v7) edge node{} (v6);
        \path[->] (v6) edge node{} (v5);
        \path[->] (v4) edge node{} (v5);
        \path[->] (v4) edge node{} (v6);

\end{tikzpicture}
    \caption{}
    \label{subfig:orientation}
\end{subfigure}\hfill
\caption{Example of a fractional solution in $\poly'(G,k)$ that satisfies all indegree inequalities, but violates a general connectivity inequality. (a) shows a graph~$G$ with vertex weights given by a fractional solution in $\poly'(G,k) \cap \{x \in \R^{n}_\geq : x \text{ satisfies } \eqref{ineq:indegree} \}$ for a fixed $c \in [k]$.
    (b) depicts an orientation $\vec E$ of $E(G)$ that maximizes the left-hand side of the indegree inequalities for that fractional solution.
    The dotted squares represent the partition $\mathcal{W}=\{\{v_1,v_2,v_3\}, \{v_4\}, \{v_5,v_6,v_7\}\}$, and thicker circles represent the vertices in $S= \{v_1,v_4,v_7\}$.}
    \label{fig:example:indegree}
\end{figure}

\begin{comment}
{\bf Separation:} Suppose that $S$ and $\mathcal{W}$ are given. Then does the separation reduce to 
\begin{align*}
\min & \sum_{v\in V(G)} x^*_{vi} \sum_{j\in [\ell]}z_{vj}\\
& y_{uv}+y_{vu}=1 & \{u,v\}\in E(\mathcal{W})\\
& z_{vj} \geq y_{uv} & j\in [\ell], u\in W_j, v\in V(G)\\
&y,z \mbox{ binary}
\end{align*}
\end{comment}

In the remainder of this section, we prove that inequalities~\eqref{ineq:cut-general} induce facets  under some hypothesis. 
First let us introduce some definitions to be used in this proof.

Let $D=(V(G), \vec E(\mathcal{W}))$.
For each $v \in V(G)$, a \emph{hanger of $v$}, denoted by $H_v$, is an inclusion-wise maximal subgraph of $D$ satisfying the following properties for every $u \in V(H_v)$:
\begin{enumerate}[start=1,label={(H\arabic*)}]
    \item there is a directed path from $u$ to $v$ in $H_v$; and
    \item the indegree of $u$ in $H_v$ is equal to $\hat d (u)$.
\end{enumerate}
Figure~\ref{fig:example:hangers} shows an example illustrating this concept.

\begin{figure}[t!]
\centering
\begin{subfigure}{.47\linewidth}
    \centering
    \begin{tikzpicture}[
auto,
            node distance = 1.5cm, semithick, scale=0.8, 
]

        \tikzstyle{every state}=[
            draw = black,
            semithick,
            fill = white,
            minimum size = 4mm,
            scale=0.9,
            inner sep=1pt
        ]

        \node[state] (v1) {$v_1$};
        \node[state] (v2) [above right of=v1] {$v_2$};
        \node[state] (v3) [below right of=v1] {$v_3$};
\node[state] (v4) [below right of=v2] {$v_4$};
        \node[state] (v5) [above right of=v4] {$v_5$};
        \node[state] (v6) [below right of=v4] {$v_6$};
        \node[state] (v7) [below right of=v5] {$v_7$};

        \node[draw,dotted,fit=(v1) (v2) (v3)] {};
        \node[draw,dotted,fit=(v5) (v6) (v7)] {};
        \node[draw,dotted,fit=(v4)] {};

        \path (v1) edge node{} (v2);
        \path (v1) edge node{} (v3);
        \path (v3) edge node{} (v2);
        \path[<-] (v4) edge node{} (v2);
        \path[<-] (v4) edge node{} (v3);
        
        \path (v7) edge node{} (v5);
        \path (v7) edge node{} (v6);
        \path (v6) edge node{} (v5);
        \path[->] (v4) edge node{} (v5);
        \path[->] (v4) edge node{} (v6);

\end{tikzpicture}\caption{Graph $G$ with an orientation of the edges crossing elements of the partition $\mathcal{W}=\{\{v_1,v_2,v_3\}, \{v_4\}, \{v_5,v_6,v_7\}\}$ that is represented by dotted squares.} 
\end{subfigure}\hfill
\begin{subfigure}{.47\linewidth}
    \centering
    \begin{tikzpicture}[
auto,
            node distance = 1.5cm, semithick, scale=0.8, 
]

        \tikzstyle{every state}=[
            draw = black,
            semithick,
            fill = white,
            minimum size = 4mm,
            scale=0.9,
            inner sep=1pt
        ]

        \node[state] (v1) {$v_1$};
        \node[state] (v2) [above right of=v1] {$v_2$};
        \node[state] (v3) [below right of=v1] {$v_3$};
\node[state] (v4) [below right of=v2] {$v_4$};
        \node[state] (v5) [above right of=v4] {$v_5$};
        \node[state] (v6) [below right of=v4] {$v_6$};
        \node[state] (v7) [below right of=v5] {$v_7$};

        \node[draw,dotted,fit=(v1) (v2) (v3)] {};
        \node[draw,dotted,fit=(v5) (v6) (v7)] {};
        \node[draw,dotted,fit=(v4)] {};

\path[<-] (v4) edge node{} (v2);
        \path[<-] (v4) edge node{} (v3);
        
\path[->] (v4) edge node{} (v5);
        \path[->] (v4) edge node{} (v6);

\end{tikzpicture}\caption{A directed graph $D=(V(G), \vec E(\mathcal{W}))$ containing six nontrivial hangers: $v_2v_4$, $v_3v_4$, $v_2v_4v_5$, $v_3v_4v_5$, $v_2v_4v_6$, and $v_3v_4v_6$.\\\hfill} 
\end{subfigure}\hfill
\caption{Example illustrating the hangers in a directed graph.}
\label{fig:example:hangers}
\end{figure}

    Let $\mathcal{H}=\{H_v\}_{v \in V(G)}$ be a collection of hangers in $D$, and define $\mathcal{W}_v = \{W \in \mathcal{W} : V(H_v)\cap W \neq \emptyset\}$ for every $v \in V(G)$.
    The collection $\mathcal{H}$ is said to be \emph{conforming} if
    \begin{enumerate}[start=1,label={(C\arabic*)}]
        \item for each $v \in V(G)$, $H_v$ is an (in)arborescence rooted at $v$ (i.e., an orientation of  rooted tree where every edge is oriented toward the root), and $|V(H_v)\cap W|\le 1$ for all $W \in \mathcal{W}$; 
\label{conforming:arborescence}
        \item for every $u,v \in V(G)$ such that $u \in V(H_v)$, it holds that $H_u \subseteq H_v$;
        \label{conforming:subgraph}
\item for each $W \in \mathcal{W}$ and $v \in W$, there is a path $P$ in $G[W]$ between~$v$ and the single vertex in $S\cap W$ such that $\mathcal{W}_u\cap \mathcal{W}_{u'} = \{W\}$ for all $u,u' \in V(P)$ with $u\neq u'$. \label{conforming:path}
    \end{enumerate}

In the example depicted in Figure~\ref{fig:example:hangers},  $\mathcal{H}_1=\{v_1, v_2, v_3, v_2v_4, v_2v_4v_5, v_2v_4v_6, v_7\}$ satisfies condition~\ref{conforming:subgraph} while $\mathcal{H}_2=\{v_1, v_2, v_3, v_2v_4, v_2v_4v_5, v_3v_4v_6, v_7\}$ does not since the hanger of $v_4$ is not contained in the hanger of $v_6$. If $S=\{v_1,v_4,v_7\}$ then $\mathcal{H}_1$ satisfies condition~\ref{conforming:path} but if $S=\{v_1,v_4,v_6\}$ then it does not since 
$\mathcal{W}_{v_5} \cap \mathcal{W}_{v_6}=\{\{v_1,v_2,v_3\}, \{v_4\}, \{v_5,v_6,v_7\}\}$.

In what follows, to prove that a face $\hat{F} = \{ x\in \poly(G,k) : \hat{\lambda}^{T}x = \hat{\lambda}_0 \}$ is a facet of
$\poly(G,k)$, we show that, if a face $F= \{ x\in \poly(G,k) :
\lambda^T x = \lambda_0 \}$ contains~$\hat{F}$, then
there exists $a \in \R$ 
such that $\lambda^T = a
\hat{\lambda}^{T}$ and $\lambda_0 = a \hat{\lambda}_0$.

\begin{theorem}\label{thm:gen-connectiviy:facetness}
     Inequality~\eqref{ineq:cut-general} induces a facet of $\poly(G,k)$ if
    \begin{enumerate}[(i)]
        \item there exists a conforming collection of hangers $\mathcal{H}=\{H_v\}_{v \in V(G)}$ in the directed graph $D=(V(G), \vec E(\mathcal{W}))$; and 
        
        \item there exist $u,v \in S$ such that $u\neq v$ and $\hat d(u)=\hat d(v)=0$ when $k\ge 2$. \label{property:zero-indegree2}
    \end{enumerate}
\end{theorem}
\begin{proof}

    The validity of inequality~\eqref{ineq:cut-general} is shown in Proposition~\ref{prop:validity:general-cut}.
    For each $v \in V(G)$, let $s(v)$ be the single vertex in $S\cap W$ where $W\in \mathcal{W}$ and $v \in W$.
    Let~$\mathcal{H}=\{H_v\}_{v \in V(G)}$ be a conforming collection of hangers in  $D=(V(G), \vec E(\mathcal{W}))$.
    For each $v \in V(G)$, $\hat N(v)$ denotes the set of inneighbors  of $v$ in $H_v$.

    Consider a vertex~$v \in V(G)$, and note that $|V(H_v)\cap W| \le 1$ for all~$W \in \mathcal{W}$ by~\ref{conforming:arborescence}.
    For every $W \in \mathcal{W}$ such that $V(H_v)\cap W =\{u\}$ for a vertex $u \in V(G)$, let $P_{W}$ be a (simple) path in $G[W]$ between $u$ and $s(u)$ where all vertices $v',u' \in V(P_W)$ with $v' \neq u'$ satisfy $\mathcal{W}_{v'}\cap \mathcal{W}_{u'} = \{W\}$.
    The existence of these paths is guaranteed by~\ref{conforming:path}. We next use these paths to connect vertices in $S$ to the hangers.

    Let $T_v$ be the (undirected) graph obtained from the union of the underlying graph of $H_v$, $P_W$  for all $W \in \mathcal{W}$ such that $V(H_v)\cap W \neq \emptyset$,  and $T_z$ for all $u \in V(P_W)\setminus(V(H_v)\cap W)$ and $z \in \hat N(u)$.
It follows from~\ref{conforming:arborescence} and the construction above that $T_v$ is a tree rooted at $v$. 
    Because of~\ref{conforming:subgraph}, there exists a collection~$\mathcal{T}=\{T_v\}_{v \in V(G)}$ such that  $T_u \subseteq T_v$ for all $u,v \in V(G)$ with $u \in V(T_v)$.
    In Figure~\ref{fig:example:tree}, we show examples of~$T_v$ obtained using the hangers depicted in Figure~\ref{fig:example:hangers}.
    
    \begin{figure}[t!]
\centering
\begin{subfigure}{.47\linewidth}
    \centering
    \centering
    \begin{tikzpicture}[
auto,
            node distance = 1.5cm, semithick, scale=0.8, 
]

        \tikzstyle{every state}=[
            draw = black,
            semithick,
            fill = white,
            minimum size = 4mm,
            scale=0.9,
            inner sep=1pt
        ]

        \node[state] (v1) [very thick]{$v_1$};
        \node[state] (v2) [above right of=v1] {$v_2$};
        \node[state] (v3) [below right of=v1] {$v_3$};
\node[state] (v4) [below right of=v2, very thick] {$v_4$};
        \node[state] (v5) [above right of=v4] {$v_5$};
        \node[state] (v6) [below right of=v4] {$v_6$};
        \node[state] (v7) [below right of=v5, very thick] {$v_7$};

        \node[draw,dotted,fit=(v1) (v2) (v3)] {};
        \node[draw,dotted,fit=(v5) (v6) (v7)] {};
        \node[draw,dotted,fit=(v4)] {};

\path (v1) edge node{} (v3);
\path (v4) edge node{} (v3);
        
        \path (v7) edge node{} (v5);
\path (v4) edge node{} (v5);

\end{tikzpicture}\caption{Tree $T_{v_5}$ formed by $H_{v_5}$, and the paths $v_1v_3$ and $v_5v_7$.
    \\\hfill} 
\end{subfigure}\hfill
\begin{subfigure}{.47\linewidth}
    \centering
    \begin{tikzpicture}[
auto,
            node distance = 1.5cm, semithick, scale=0.8, 
]

        \tikzstyle{every state}=[
            draw = black,
            semithick,
            fill = white,
            minimum size = 4mm,
            scale=0.9,
            inner sep=1pt
        ]

        \node[state] (v1) [very thick]{$v_1$};
        \node[state] (v2) [above right of=v1] {$v_2$};
        \node[state] (v3) [below right of=v1] {$v_3$};
\node[state] (v4) [below right of=v2, very thick] {$v_4$};
        \node[state] (v5) [above right of=v4] {$v_5$};
        \node[state] (v6) [below right of=v4] {$v_6$};
        \node[state] (v7) [below right of=v5, very thick] {$v_7$};

        \node[draw,dotted,fit=(v1) (v2) (v3)] {};
        \node[draw,dotted,fit=(v5) (v6) (v7)] {};
        \node[draw,dotted,fit=(v4)] {};

        \path (v1) edge node{} (v2);
\path (v3) edge node{} (v2);
\path (v4) edge node{} (v3);
        
\path (v7) edge node{} (v6);
\path (v4) edge node{} (v6);

\end{tikzpicture}\caption{Tree $T_{v_6}$ formed by $H_{v_6}$, and the paths $v_1v_2v_3$ and $v_6v_7$. 
    It contains, for example, $T_{v_3}=v_1v_2v_3$ and $T_{v_2}=v_1v_2$. }
\end{subfigure}\hfill
\caption{Examples of trees $T_v$ formed by hangers depicted in Figure~\ref{fig:example:hangers}. Thicker circles identify the vertices in $S=\{v_1, v_4, v_7\}$.}
    \label{fig:example:tree}
\end{figure}

    Let $\hat{F} = \{ x\in \poly(G,k) : \hat{\lambda}^{T}x = \hat{\lambda}_0 \}$ be a face of $\poly(G,k)$ where $(\hat{\lambda}^{T}, \hat{\lambda}_0)$ corresponds to inequality~\eqref{ineq:cut-general} for a fixed $c \in [k]$.
    Suppose that $F= \{ x\in \poly(G,k) : \lambda^T x = \lambda_0 \}$ is a face of $\poly(G,k)$ that contains~$\hat{F}$.
    We shall prove that $\hat F$ is a facet of $\poly(G,k)$ by  induction on the height of the trees in $\mathcal{T}$.
    For every~$i \in \Z_\geq$, let us denote by $\mathcal{T}_i$ the (possibly empty) subset  of~$\mathcal{T}$ containing the trees of height equal to $i$.
    
    First note that there must exist a tree of height zero in $\mathcal{T}$ as $\mathcal{H}$ is conforming.
    Let $T_v \in \mathcal{T}$ be a tree of height zero, that is, it has a single vertex~$v$.
    By the definitions of $H_v$ and $T_v$, it holds that $\hat d(v)=0$  and $ v \in S$.
Hence, the vector $e(v, c)$ belongs to~$\hat F$, and so $\lambda_{v, c} = \lambda_0$.
    Suppose now $k\ge 2$.
    For every $b \in [k]\setminus\{c\}$ and $u \in V(G)\setminus\{v\}$, the vectors $e(v,c)$ and $e(v,c) + e(u,b)$ belong to $\hat F$, which implies $\lambda_{u,b}=0$.
    To compute the remaining zero coefficients, note that by~\ref{property:zero-indegree2} there is $u \in S\setminus \{v\}$ such that $\hat d(u)=0$.
    As a consequence, for every $b \in [k]\setminus\{c\}$, $e(u,c)$ and $e(u,c)+e(v, b)$ belong to~$\hat F$, and thus $\lambda_{v, b} = 0$.
    This completes the proof of the base case of our induction.

    Let $j \in \Z$ with $j \ge 1$ and $\mathcal{T}_j \neq \emptyset $, and
    suppose that, for all $i \in  \{0, \ldots, j-1\}$,
we have $e(T, c) \in \hat F$ (equivalently,  $\hat \lambda^T e(T,c)=1$) for every $T \in \mathcal{T}_i$.
We next compute the non-zero entries of~$\lambda$.
    Consider a tree $T_v \in \mathcal{T}_j$.
    If $v \in S$, then
    $e(T_{v}, c) = e(v, c)+\sum_{u \in \hat N(v)} e(T_u,c)$.
By the induction hypothesis, it follows that 
    \[\hat \lambda^T e(T_{v}, c) = (1-\hat d(v)) + \sum_{u \in \hat N(v)} \hat \lambda^T e(T_u,c) = (1-\hat d(v)) + \hat d(v) = 1.\]
    Thus, $e(T_{v}, c) \in \hat F$, and we conclude that $\lambda_{v,c} = (1-  \hat d(v)) \lambda_0$ since $\lambda^T e(T_{v},c) = \lambda_{v,c} + \sum_{u \in \hat N(v)} \lambda^T e(T_u,c) = \lambda_{v,c} + \sum_{u \in \hat N(v)} \lambda_0=\lambda_{v,c} + \hat d(v) \lambda_0=\lambda_0$.

    Suppose now that $v \notin S$.
By the construction of $T_v$,  we have
    \begin{equation}\label{eq:non-sources}
    e(T_v,c) = e\left(T_{s(v)}, c\right) + \sum_{u \in V(P)\setminus\{s(v)\}} \left( e(u,c) + \sum_{z \in \hat N(u)} e(T_z, c) \right),
    \end{equation}
where $P$ is the path in $T_v$ linking $v$ and $s(v)$.
For each $u \in V(P)\setminus\{s(v)\}$, it follows from the induction hypothesis  that 
    \[\hat \lambda^T e(u,c) + \sum_{z \in \hat  N(u)} \hat \lambda^T e(T_z,c)= -\hat d(u) + \sum_{z \in \hat N(u)} 1=-\hat d(u) + \hat d(u) =0.\] 
    It follows from the previous equation, equation~\eqref{eq:non-sources}, and the induction hypothesis that
    \(\hat \lambda^T e(T_v,c) = \hat \lambda^T e(T_{s(v)}, c) = 1.\)
    Hence $e(T_v,c) \in \hat F$.
Let $v'$ be the vertex adjacent to $v$ in the path between $v$ and $s(v)$ in $T_v$. 
    Observe now that 
    \(e(T_v, c) = e(v,c) + \sum_{u \in \hat N(v)} e(T_u, c) + e(T_{v'},c).\)
As a consequence, we have 
    \[\lambda^T e(T_v,c) = \lambda^T e(v,c) + \sum_{u \in \hat N(v)} \lambda^T e(T_u, c) + \lambda^T e(T_{v'},c) = \lambda_{v,c} + \hat d(v)\lambda_0 +\lambda_0 =\lambda_0.\]
    Thus $\lambda_{v,c}=-\hat d(v)\lambda_0$.
Therefore, it holds that $\lambda = \lambda_0 \hat \lambda$, and so inequality~\eqref{ineq:cut-general} induces a facet of $\poly(G,k)$.
 \end{proof}

To conclude this section, we observe that the properties in the statement of Theorem~\ref{thm:gen-connectiviy:facetness} are also satisfied by the facet-inducing inequalities of~\eqref{ineq:cut} and~\eqref{ineq:indegree}.

First let $Z$ be  a minimal $u,v$-separator of a pair of non-adjacent vertices $u,v \in V(G)$.
Consider now the corresponding connectivity inequality~\eqref{ineq:cut-general} obtained as in the proof of Proposition~\ref{prop:connectivity}, and note that one can easily find a collection of hangers satisfying~\ref{conforming:arborescence} and~\ref{conforming:subgraph} as all edges are oriented toward vertices in $Z$, and $\hat d(u)= \hat d(v) = 0$.  Condition~\ref{conforming:path} is satisfied by any collection of hangers in this orientation since, for every vertex in the class containing $v$, there is a path from this vertex to $v$ which contains at most one vertex in $Z$ as this separator is minimal.

For the facet-defining indegree inequalities~\eqref{ineq:indegree}, 
note that by Corollary~\ref{cor:simple-facets}\ref{cor:item:indegree} the orientation is acyclic, and so, for each vertex, there is only one possible hanger, and it is an arborescence. 
This implies condition~\ref{conforming:subgraph}.
Moreover, \ref{conforming:arborescence} and \ref{conforming:path} are satisfied as every class in the partition has size one.

 \section{Separation of generalized connectivity inequalities}\label{sec:separation-single}

Let $n$ and $m$ denote the number of vertices and edges of $G$, respectively.
The separation problems associated with the connectivity inequalities~\eqref{ineq:cut} on input~$(G,k)$ can be reduced to the \textsc{Minimum Vertex Cut} problem which can be solved using Orlin's Max Flow algorithm  in time $\bigO(kn^3 m)$.
Regarding the indegree inequalities~\eqref{ineq:indegree}, for each~$c \in [k]$, the associated separation problem can be solved in $\bigO(n+m)$ time by simply orienting every edge $\{u,v\} \in E(G)$ from $u$ to $v$ if~$\bar x_{u,c}>\bar x_{v,c}$, where $\bar x \in \R^{nk}_\geq$,  as proposed by Wang et al.~\cite{wang2017imposing}.

The remainder of this section is devoted to the design of an algorithm to solve the separation problem associated with inequalities~\eqref{ineq:cut-general} in polynomial time for a given (fixed) partition of $V(G)$.
Let~$\mathcal{W}=\{W_i\}_{i  \in [\ell]}$ be a partition of $V(G)$.
Given a vector $x^* \in \R^{nk}_\geq$  and $c \in [k]$, the separation problem associated with~\eqref{ineq:cut-general} (for the partition $\mathcal{W}$) reduces to finding an orientation of the edges in $E(\mathcal{W})$ that minimizes \(\sum_{v \in V(G)} \hat d(v) x^*_{v,c}\) as the left-hand side of~\eqref{ineq:cut-general} is equal to 
\[\sum_{v \in S} x_{v,c} - \sum_{v \in V(G)} \hat d(v) x_{v,c}\]
and the choice of $S$ is trivial.
We next show how to reduce this problem to $\bigO(n^2)$ bipartite instances of the \textsc{Minimum Weighted Vertex Cover} problem (VC).
Given a graph  $H$ and a weight function $w \colon V(H) \to \Q_\geq$,  VC consists in finding a set $C \subseteq V(H)$ such that $\{u,v\}\cap C\neq \emptyset$ for every edge $\{u,v\} \in E(H)$, and $\val(C):=\sum_{v \in C} w(v)$ is minimum.
This problem can be solved in polynomial time on bipartite graphs by computing maximum flows (see~Chlebík and  Chlebíková~\cite{CHLEBIK2008292}).

The core of the separation procedure is described in Algorithm~\ref{alg:separation}, which uses a subroutine called $\textsc{MinW-VC-Bipartite}$ for computing (in polynomial time) a minimum-weight vertex cover of a bipartite graph.

\begin{algorithm}[t!]
  \begin{algorithmic}[1]
    \Require{Partition $\mathcal{W}=\{W_i\}_{i \in [\ell]}$, class $c \in [k]$, vector $x^* \in \R^{nk}_\ge$}
    \Ensure{Orientation $\vec E$ of $E(\mathcal{W})$}
    
    \State $ \vec E \gets \emptyset$
    \ForAll{$v \in V(G)$}    $w(v) \gets x^*_{v,c}$ \label{line:w}
    \EndFor
    \For{$i=1,\ldots,\ell-1$} \label{line:loop-i}
        \For{$j=i+1, \ldots, \ell$} \label{line:loop-j}
\State Let $G_{ij}$ be the subgraph of $G$ induced by the edges $\{\{u,v\} \in E(\mathcal{W}) : u \in W_i \text{ and } v \in W_j\}$ \label{line:bip}
\State $C_{ij} \gets $\textsc{MinW-VC-Bipartite}$(G_{ij},w_{ij})$ where $w_{ij}$ is $w$ restricted to $V(G_{ij})$\label{line:opt-vc}
\State Let $\vec E(W_i,W_j)$ be an orientation of $E(G_{ij})$ where each edge is oriented toward an endpoint in $C_{ij}$ \label{line:orient}
            \State $\vec E \gets \vec  E \cup \vec E(W_i,W_j)$ 
        \EndFor
    \EndFor
\State \Return $\vec E$
  \end{algorithmic}
  \caption{\textsc{Separation-Connectivity}($\mathcal{W}$, $c$, $x^*$)}
  \label{alg:separation}
\end{algorithm}

\begin{theorem}
Let $\ell \in \Z$ be positive.
The separation problem associated with inequalities~\eqref{ineq:cut-general} for a given $\ell$-partition of $V(G)$ can be solved in $\bigO(knm)$ time using Algorithm~\ref{alg:separation}.
\end{theorem}
\begin{proof}
    Let $\mathcal{W}=\{W_i\}_{i \in [\ell]}$ be an $\ell$-partition of $V(G)$, $c \in [k]$, and $x^* \in \R^{nk}_\ge$.
    Firstly consider any orientation $\vec E(\mathcal{W})$ of the edges in $E(\mathcal{W})$, and the function~$\hat d\colon V(G) \to [\ell]\cup\{0\}$ corresponding to such orientation.
    For each $i,j \in [\ell]$ with $i<j$, and each $v \in V(G)$, we define $d_{ij}(v)=1$ if there exists an arc between $W_i$ and $W_j$ with head $v$, and $d_{ij}(v)=0$ otherwise.
    Note that $d_{ij}(v)=0$ if $v \notin W_i\cup W_j$, and that $\hat d(v) = \sum_{i,j \in [\ell] : i < j} d_{ij}(v)$.
    Let us define $C_{ij} = \{v \in V(G) : d_{ij}(v)=1\}$ for all $i,j \in [\ell]$ with $i<j$.
    One may easily check that $C_{ij}$ is a vertex cover of the edges in $E(\mathcal{W})$ with extremities in $W_i$
    nd $W_j$.
    Assuming that $x^*_{v,c}$ corresponds to the weight of vertex $v$ for every $v \in V(G)$,  it holds that $\val(C_{ij}) = \sum_{v \in V(G)} d_{ij}(v) x^*_{v,c}$.
    As a consequence, we have the following sequence of equations:
    \begin{equation}\label{eq:cover-orientation}
    \sum_{v \in V(G)} \hat d(v) x^*_{v,c} = \sum_{i,j \in [\ell] : i < j} \sum_{v \in V(G)}  d_{ij}(v) x^*_{v,c}  = \sum_{i,j \in [\ell] : i < j} \val(C_{ij}). 
    \end{equation}
    
    Let $i,j \in [\ell]$ with $i<j$.
    Consider now an instance~$(G_{ij}, w)$ of VC consisting of a bipartite graph $G_{ij}$ with bipartition $\{W_i, W_j\}$ and edges $\{\{u,v\} \in E(\mathcal{W}) : u \in W_i \text{ and } v \in W_j\}$ (line~\ref{line:bip} of Algorithm~\ref{alg:separation}), and a weight function $w_{ij}(v)=x^*_{v,c}$ for all $v \in V(G_{ij})$ (line \ref{line:w}).
    Let $C^*_{ij}$ be a minimum weight vertex cover of $(G_{ij}, w_{ij})$ (line~\ref{line:opt-vc}), 
and define an orientation $\vec E(W_i,W_j)$ such that every edge is oriented toward a vertex in $C^*_{ij}$ (line~\ref{line:orient}).
    We remark that, for each $v \in C^*_{ij}$ with $w_{ij}(v)>0$, there is a vertex $u$ adjacent to $v$ such that $ u \notin C^*_{ij}$ due to the choice of the vertex cover.
    Hence the vertices in~$C^*_{ij}$ with positive weight are head of some arc in~$\vec E(W_i,W_j)$.
    It follows from this remark and previous definitions that $\val(C^*_{ij}) = \sum_{v \in V(G)} d_{ij}(v) x^*_{v,c}$, where~$d_{ij}$ is defined with respect to $\vec E(W_i,W_j)$.
    
    By equation~\eqref{eq:cover-orientation}, the orientation of $E(\mathcal{W})$ obtained from optimal solutions of the vertex cover instances $(G_{ij},w_{ij})$, for all $i,j \in [\ell]$ with $i<j$, gives the minimum possible value of $\sum_{v \in V(G)} \hat d(v) x^*_{v,c}$.
    This proves that Algorithm~\ref{alg:separation} produces an optimal orientation $E(\mathcal{W})$.
    To conclude the separation, it suffices noticing that, for each $i \in [\ell]$, one may choose a vertex in $v_i \in \arg \max_{v \in W_i} x^*_{v,c}$, and take $S = \{v_i\}_{i \in \ell}$.

     Note that \textsc{MinW-VC-Bipartite} can be executed in 
$\bigO(|V(G_{ij})||E(G_{ij})|)$ using Orlin's maximum flow
    algorithm~\cite{ORLIN2013} for all $i,j\in [\ell]$ with $i<j$.
    Therefore, Algorithm~\ref{alg:separation} runs in $\bigO(nm)$ time for each $c \in [k]$, and so the separation of~\eqref{ineq:cut-general} can be performed in $\bigO(knm)$ time. 
 \end{proof}

We next observe that the running time of Algorithm~\ref{alg:separation} equals the running times of the separation routines described in the literature for the connectivity and indegree inequalities.

Considering the connectivity inequalities~\eqref{ineq:cut},  Algorithm~\ref{alg:separation} with $\ell=2$ is essentially the same algorithm described in the literature for separating these inequalities by computing a maximum flow in a network (see Fischetti et al.~\cite{Fischetti2017}).

Observe now that the loops in lines~\ref{line:loop-i} and~\ref{line:loop-j} of Algorithm~\ref{alg:separation} can be replaced by a loop on
each $W \in \mathcal{W}$ and each edge with exactly one endpoint in $W$.  
In the case $\ell=n$, the minimum vertex cover in line~\ref{line:opt-vc} can be trivially computed in constant time (simply choose the vertex of smallest weight).
Therefore, Algorithm~\ref{alg:separation} can be implemented to run in $\bigO(n+m)$ time when $\ell=n$, which corresponds to solving the separation problem of the indegree inequalities~\eqref{ineq:indegree} (for any given class $c \in [k]$) in linear time.
That is the same time complexity of the separation algorithm proposed by Wang et al.\cite{wang2017imposing}.

We next show that Algorithm~\ref{alg:separation} can be implemented to run in linear time even when $\ell <n$, provided that all edges in $E(\mathcal{W})$ have endpoints in \emph{singleton} classes, that is, classes in $\mathcal{W}$ with cardinality equal to one.
We say that a partition $\mathcal{W}$ of $V(G)$ is \emph{nice} if every edge in $E(\mathcal{W})$ has an endpoint in a singleton class.

\begin{theorem}
The separation problem of inequalities~\eqref{ineq:cut-general} can be solved in $\bigO(k(n+m))$ time on nice partitions using Algorithm~\ref{alg:separation}.
\end{theorem}
\begin{proof}
    To prove the claimed result, it suffices to show that every execution of line~\ref{line:opt-vc} in Algorithm~\ref{alg:separation} takes $\bigO(|E(G_{ij})|)$ time if either $W_i$ or $W_j$ is a singleton class.
    Assume, without loss of generality, that $W_i=\{u\}$ for a vertex $u \in V(G)$.
    Since $w_{ij}(v)\geq 0$ for all~$v \in V(G_{ij})$, and $G_{ij}$ has no isolated vertex, one may easily conclude that  either the set $\{u\}$, or $W_j$ is an optimal vertex cover of $G_{ij}$, and that such a solution can be computed in $\bigO(|E(G_{ij})|)$.
 \end{proof} \section{Multiclass inequalities}
\label{sec:multi}
We next derive strong valid inequalities for $\poly(G,k)$ that combine different classes.
In contrast to the single-class inequalities, the multiclass inequalities may have non-zero coefficients associated with any subset of classes $C \subseteq[k]$.
As we shall see, the proposed inequalities generalize~\eqref{ineq:cut} in different ways.
The ones in Section~\ref{subsec:multiway} are based on stable sets in $G$ instead of pairs of non-adjacent vertices (i.e., stable sets of size two) in $G$.
The inequalities in Section~\ref{subsec:pair} are induced by a collection of pairs of non-adjacent vertices.

\begin{comment}could we get something that generalizes both  families of inequalities?
\begin{proposition
    Let $C \subseteq [k]$ with $|C|\ge 1$, $Z$ be a separating set of $G$ and $V_c \subseteq  V(G)\setminus Z$  such that $|V_c\cap V(K)|\le 1$ for every~$K \in \mathcal{K}$.
The following inequality is valid for $\poly(G,k)$:
\begin{equation}\label{ineq:multicolor-cut}
       \sum_{c \in C} \sum_{v\in V_c}  x_{v,c} - \sum_{z\in Z} \sum_{c \in C} \beta_{zc} x_{z,c} \leq |C|,
    \end{equation}
    where $\beta_{zc} := ???\max\{|S|-|C|, \: 0\}$.
 \end{proposition}
These will be inequalities 60-65 in the porta output for a path with 5 vertices. 38-59 are inequalities (6). 
\\
In (60), $V_1=\{3,5\}$ and $V_2=\{1,3,5\}$, $Z=\{2,4\}$, 
$\beta_{41}=1$, $\beta_{21}=0$, $\beta_{42}=0$, and $\beta_{22}=1$. 
\\
In (62), $V_1=\{1,3\}$ and $V_2=\{1,3,5\}$, $Z=\{2,4\}$, 
$\beta_{41}=0$, $\beta_{21}=1$, $\beta_{42}=1$, and $\beta_{22}=0$. 
\\
In  (64), $V_1=\{1,5\}$ and $V_2=\{1,3,5\}$, $Z=\{2,4\}$, 
$\beta_{41}=0$, $\beta_{21}=0$, $\beta_{42}=1$, and $\beta_{22}=1$. The other three inequalities are symmetric. } 
\end{comment}

\subsection{Multiway inequalities}\label{subsec:multiway}

First we present the definition of multiway cuts, which generalizes the concept of separators.
Given a stable set $S \subset V(G)$,  a \emph{multiway cut} of $S$ in $G$ is set of vertices $Z \subseteq V(G)\setminus S$ such that there is no path in $G-Z$ between any pair of distinct vertices in $S$, that is, each component of $G-Z$ contains at most one vertex in $S$.
Let us denote by $\mathcal{K}$ the set of components of $G-Z$.

\begin{proposition}\label{prop:multicolor-cut-validity}
    Let $C \subseteq [k]$ with $|C|\ge 1$, $S \subseteq V(G)$ be a stable set of~$G$, $Z$ be a multiway cut of $S$ in $G$, and $\beta := \max\{|S|-|C|, \: 0\}$.
    The following inequality is valid for~$\poly(G,k)$.
    \begin{equation}\label{ineq:multicolor-cut}
        \sum_{v\in S} \sum_{c \in C} x_{v,c} - \sum_{z\in Z} \sum_{c \in C} \beta x_{z,c} \leq |C|.
    \end{equation}
\end{proposition}
\begin{proof}
We assume that $|S|>|C|$, otherwise the inequality holds trivially. 
Consider a connected $k$-subpartition $\mathcal{V} = \{V_i\}_{i \in C}$ of $G$ and let $\bar x = \chi(\mathcal{V})$.
    If $|V_i \cap S| \leq 1$ for every $i \in C$, then 
    \(\sum_{v\in S} \sum_{c \in C} x_{v,c} \leq |C|\) and~\eqref{ineq:multicolor-cut} is satisfied by $\bar x$.
    Fix now $j \in C$ such that $|V_j \cap S| > 1$.
    Since $G[V_j]$ is connected, we have $Z\cap V_j \neq \emptyset$ and hence
    \[\sum_{v\in S} \sum_{c \in C} \bar x_{v,c} - \sum_{z\in Z} \sum_{c \in C} \beta \bar x_{z,c} \leq \sum_{v\in S} \sum_{c \in C} \bar x_{v,c} -|S| +|C| \le |C|,\]    
    where the last inequality holds because $\sum_{v\in S} \sum_{c \in C} \bar x_{v,c} \leq |S|$.
    Therefore, inequality~\eqref{ineq:multicolor-cut} is valid for $\poly(G,k)$.
 \end{proof}

Figure~\ref{fig:example:multicolor-cut} depicts two examples of inequalities~\eqref{ineq:multicolor-cut} .

\begin{figure}[t!]
\centering
\begin{subfigure}{.47\linewidth}
    \centering
    \begin{tikzpicture}[
auto,
            node distance = 1.5cm, semithick, scale=0.8, 
]

        \tikzstyle{every state}=[
            draw = black,
            semithick,
minimum size = 4mm,
            scale=0.9,
            inner sep=1pt
        ]

        \node[state] (u) {$v_4$};
\node[state, very thick] (v1) [above right of=u] {$v_1$};
        \node[state, very thick] (v2) [right  of=u] {$v_2$};
        \node[state, very thick] (v3) [below  right of =u] {$v_3$};

        \node[draw,dotted, fit=(u)] {};

        \path (u) edge node{} (v1);
        \path (u) edge node{} (v2);
        \path (u) edge node{} (v3);

\end{tikzpicture}\caption{$\sum_{i=1}^3 (x_{v_i,a}+x_{v_i,b}) - x_{v_4,a} - x_{v_4,b} \leq 2$\\} 
    \label{subfig:multicolor-cut:star}
\end{subfigure}\hfill
\begin{subfigure}{.47\linewidth}
    \centering
    \begin{tikzpicture}[
auto,
            node distance = 1.5cm, semithick, scale=0.8, 
            inner sep=1pt
]
\tikzstyle{every state}=[
            draw = black,
            thick,
minimum size = 4mm,
            scale=0.9
        ]

        \node[state] (v4) {$v_4$};
        \node[state, very thick] (v1) [above right of=v4] {$v_1$};
        \node[state, very thick] (v2) [below right of=v4] {$v_2$};
        \node[state, very thick] (v3) [below left of = v4] {$v_3$};
\node[state] (v5) [below right of=v1] {$v_5$};

        \node[draw,dotted,fit=(v4) (v5)] {};

        \path (v4) edge node{} (v1);
        \path (v4) edge node{} (v2);
        \path (v4) edge node{} (v3);
        \path (v5) edge node{} (v1);
        \path (v5) edge node{} (v2);
        
    \end{tikzpicture}
    \caption{$\sum_{i=1}^3 (x_{v_i,a}+x_{v_i,b}) - \sum_{i=4}^5 (x_{v_{i},a} + x_{v_{i},b}) \leq 2$}
    \label{subfig:multicolor-cut:c4}
\end{subfigure}\hfill
\caption{Examples of inequalities~\eqref{ineq:multicolor-cut}.
Thicker vertices identifies the stable set,  dotted rectangles represent the corresponding multiway cuts in the graphs and $C=\{a,b\}$.}
    \label{fig:example:multicolor-cut}
\end{figure}

We remark that inequalities~\eqref{ineq:multicolor-cut} contain  inequalities~\eqref{ineq:cut} as every separator of two non-adjacent vertices is also a multiway cut.
Consider the star depicted in Figure~\ref{subfig:multicolor-cut:star}.
Let $\bar x \in \R^{8}$ such that  $\bar x_{v_i,a} = \bar x_{v_i,b} =1/2$  for $i \in \{1,2,3\}$ and $\bar x_{v_4,a}= \bar x_{v_4,b}=1/4$ with classes $a$ and $b$.
Clearly, $\bar x$ satisfies \eqref{ineq:cover} and \eqref{ineq:cut}.
Indeed, inequality~\eqref{ineq:multicolor-cut} when $C=\{a\}$ or $C=\{b\}$  hold for $\bar x$.
On the other hand, it holds that $\sum_{i=1}^3 (\bar x_{v_i,a} + \bar x_{v_i,b}) - (\bar x_{v_4,a}+\bar x_{v_4,b}) = 3 - 1/2 > 2$, and thus inequality~\eqref{ineq:multicolor-cut} for~$C=\{a,b\}$ is violated.
This simple example shows that, considering multiple classes, one may cut fractional solutions that satisfy~\eqref{ineq:multicolor-cut} with a single class.

Wang et al.~\cite{wang2017imposing} proved that inequalities~\eqref{ineq:cut} are facet-defining if, and only if, the separator~$Z$ is minimal.
We shall see that this is no longer a necessary condition for~\eqref{ineq:multicolor-cut} to be facet-defining when more than one class is considered, i.e., when  $|C|\ge 2$.
Figure~\ref{subfig:multicolor-cut:c4} depicts an inequality~\eqref{ineq:multicolor-cut} where $Z$ is not minimal, and yet the inequality induces a facet.

Let $C$, $S$, and $Z$ be the sets defining an inequality~\eqref{ineq:multicolor-cut} as in the statement of Proposition~\ref{prop:multicolor-cut-validity}.
We say that the triple $(C,S,Z)$ is \emph{perfect} if
\begin{enumerate}[(i)]
        \item $|S|>|C|$; and \label{perfectness:size}
        \item for every $z \in Z$, $|\mathcal{K}_z| \geq |S|-|C|+1$, where $\mathcal{K}_z = \{K \in \mathcal{K} : N(z)\cap V(K) \neq \emptyset \text{ and }  S\cap V(K) \neq \emptyset\}$; and \label{perfectness:separating}
\item there exists $z \in Z$ with $|\mathcal{K}_z| \geq |S|-|C|+2$  if $|C|\ge2$.\label{perfectness:separating-2}
\end{enumerate}

\begin{lemma}\label{lemma:multicolor-cut-facetness}
    If inequality~\eqref{ineq:multicolor-cut} induces a facet of $\poly(G,k)$, then
    $(C,S,Z)$ is perfect.
\end{lemma}
\begin{proof}
    Let $\hat{F}$ be the face of $\poly(G,k)$ induced by inequality~\eqref{ineq:multicolor-cut}.
    If~$|S|\leq |C|$, then inequality~\eqref{ineq:multicolor-cut} is dominated by the sum of inequality~\eqref{ineq:cover} for all $v \in S$, and so $\hat F$ is not a facet of $\poly(G,k)$.
    In the remainder of this proof, we shall use that
    $\sum_{z\in Z} \sum_{c \in C} x_{z,c} \leq 1$ for all $x \in \hat F$ since $|S|>|C|$, and $\sum_{v\in S} \sum_{c \in C} x_{v,c} \leq |S|$ is valid for $\poly(G,k)$.
    
    Consider  a vector $\bar x \in \hat F \cap \{0,1\}^{nk}$.
    Suppose now that $|S| > |C|$, and there exists a vertex $u \in Z$ such that $|\mathcal{K}_u| \leq |S|-|C|$.
    If $\bar x_{u,t}=1$ for some $t \in C$, then $\bar x_{z,c}=0$ for all $z  \in Z\setminus\{u\}$ and $c \in C$.
    Hence each class $c \in C\setminus \{t\}$ contains at most one vertex of $S$, and so $\sum_{v\in S} \sum_{c \in C\setminus\{t\}} \bar x_{v,c} \leq |C|-1$.
    This implies $\sum_{v\in S} \bar x_{v,t} \geq |S|-|C|+1$, a contradiction to $|\mathcal{K}_u| \leq |S|-|C|$.
    Hence it holds that $x_{u,t}=0$ for all $x \in \hat F$, and $\hat F$ is not a facet.

    Suppose now that $(C,Z,S)$ satisfies all conditions for perfectness but~\ref{perfectness:separating-2}.
    This implies $|C|\ge 2$, and $|\mathcal{K}_z|=|S|-|C|+1$ for all $z \in Z$.
If a vertex in $Z$ is assigned a class $t \in C$, then every class in $C\setminus\{t\}$ contains precisely one vertex of $S$, and class $t$ contains precisely $|S|-|C|+1$ vertices of $S$.
    In other words, it holds that $\sum_{v\in S} \bar x_{v,t} = |S|-|C|+1$, and $\sum_{v\in S} \bar x_{v,c} = 1$ for all $c \in C\setminus\{t\}$.
    If no vertex in $Z$ is assigned a class in $C$, then every class in $C$ must contain exactly one vertex of S, and so $\sum_{v \in S} \bar x_{v,c} = 1$ for all $c \in C$.
    The discussion above shows that, for every $t \in C$,  $\bar x$ satisfies the following equations:
    \begin{align*}
        \sum_{v \in S} \bar x_{v,c} =1 & \text{ for all } c \in C\setminus \{t\}. 
    \end{align*}
Therefore, these equations hold for every point in $\hat F$, and so it is not a facet of~$\poly(G,k)$.
 \end{proof}

\begin{theorem}
    Inequality~\eqref{ineq:multicolor-cut} induces a facet of $\poly(G,k)$ if and only if $(C,S,Z)$ is perfect.
\end{theorem}
\begin{proof}

    Lemma~\ref{lemma:multicolor-cut-facetness} proves that perfectness is a necessary condition for~\eqref{ineq:multicolor-cut} to induce a facet of~$\poly(G,k)$.
    The remainder of this proof is devoted to showing that perfectness is also sufficient.
    
    Let $\hat{F} = \{ x\in \poly(G,k) : \hat{\lambda}^{T}x = \hat{\lambda}_0 \}$ be a face of $\poly(G,k)$ where $(\hat{\lambda}^{T}, \hat{\lambda}_0)$ corresponds to inequality~\eqref{ineq:multicolor-cut}.
    Suppose that $F= \{ x\in \poly(G,k) : \lambda^T x = \lambda_0 \}$ is a face of $\poly(G,k)$ which contains~$\hat{F}$.
    To compute the zero coefficients associated with classes not in $C$, for each~$v \in V(G)$, consider a set $S_v \subseteq S\setminus \{v\}$ with $|S_v|=|C|$ and any bijection $\pi_v \colon S_v \to C$.
    Note that, for each $t \in [k]\setminus C$, the vectors $e(v,t)+\sum_{u \in S_v} e(u, \pi_v(u))$ and $\sum_{u \in S_v} e(u, \pi_v(u))$ belong to $\hat F$ as they correspond to connected $k$-subpartitions where every class in $C$ contains exactly one vertex of~$S$.
    This implies $\lambda_{v,t} = 0$ for all $v \in V(G)$ and $t \in [k]\setminus C$.

    Let $\mathcal{K}' = \{K \in \mathcal{K} : S\cap V(K) \neq \emptyset\}$. 
    For each $v \in S$, let $\mathcal{K}_v \subset \mathcal{K}'$ such that  $|\mathcal{K}_v|=|C|-1$ and $v \notin V(K)$ for all $K \in \mathcal{K}_v$ 
    (it exists since $|\mathcal{K}'|=|S|>|C|$).
    Moreover, let $W_v$ be the component of $G-Z$ which contains $v$.
    Consider~$t \in C$, and any bijection $\pi_v \colon \mathcal{K}_v \to C\setminus \{t\}$. 
As $W_v$ is connected, there is a sequence $G_1, \ldots, G_{|W_v|}$ of connected subgraphs of $W_v$ such that $v$ is the single vertex of~$G_1$, and $G_p = G_{p+1}-v_{p+1}$  where $v_{p+1} \in V(G_{p+1})$ for each $p \in \left [|W_v|-1 \right]$. 
    For every $p \in [|W_v|-1]$, the vectors
\[e(G_p, t) + \sum_{K \in \mathcal{K}_v} e(K, \pi_v(K)) \quad \text{and} \quad e(G_{p+1}, t) + \sum_{K \in \mathcal{K}_v} e(K, \pi_v(K))\] 
    belong to $\hat F$ since these vectors induce connected $k$-subpartitions where each class in $C\setminus\{t\}$ corresponds to a component in $\mathcal{K}_v$, and so it contains exactly one vertex of $S$.
Thus $\lambda_{u,t}=0$ for all $u \in V(K)\setminus S$ with $K \in \mathcal{K}'$, and $t \in C$.

    Suppose that $G-Z$ has a component $K$ such that $V(K)\cap S=\emptyset$.
    As $G$ is connected, there exist vertices $z \in Z$ and $z' \in N(z)\cap V(K)$.
Let $G_1, \ldots, G_{|V(K)|}$ be a sequence of connected subgraphs of $K$ such that~$z'$ is the single vertex of~$G_1$, and $G_p = G_{p+1}-v_{p+1}$  where $v_{p+1} \in V(G_{p+1})$ for each $p \in [|V(K)|-1]$. 
    Let $\{W_i\}_{i \in [q]}$ be a collection of $q = |S|-|C|+1$ distinct components in $\mathcal{K}_z$.
    Such a collection exists because of property~\ref{perfectness:separating} of perfectness.
    It follows from this construction that the graph $H_0$ induced by $\bigcup_{i \in [q]} V(W_i) \cup \{z\}$ is connected.
    Observe now that, for each $p \in [|V(K)|]$, the subgraph $H_p$ of $G$ induced by $V(H_0) \cup  V(G_p)$ is also connected.
    We next enumerate vectors corresponding to connected subpartitions that contain the vertices of $H_p$ with $p \in [|V(K)|]\cup \{0\}$ as one of their classes.

    Let $\mathcal{K}'' = \mathcal{K'}\setminus \{W_i\}_{i \in [q]}$ and note that $|\mathcal{K}''|=|C|-1$.
For every~$t \in C$ and $p \in \{0,1, \ldots, |V(K)|\}$, one can verify that the vector \[e(H_p, t) + \sum_{K \in \mathcal{K}''} e(K, \sigma_{t}(K))\]
belongs to $\hat F$, where $\sigma_{t} \colon \mathcal{K}'' \to C \setminus\{t\}$ is a bijection.
    Hence, for every component~$K$ of $G-Z$ such that $V(K)\cap S = \emptyset$, we have $\lambda_{u,t}=0$ for every $u \in V(K)$ and $t \in C$.

    We next compute the non-zero coefficients of the inequality $(\lambda,\lambda_0)$.
    Consider $u,v \in S$ with $u\neq v$, and $t \in C$.
    Since $|S|>|C|$, there is a set $S' \subseteq S\setminus \{u,v\}$ such that $|S'| = |C|-1$.
    Let $\pi \colon S' \to C\setminus \{t\}$ be a bijection.
    The vectors $e(u, t) + \sum_{w \in S'} e(w,\pi(w))$ and $e(v, t) +  \sum_{w \in S'} e(w,\pi(w))$ are in $\hat F$ as precisely~$|C|$ vertices of $S$ are assigned to singleton classes indexed by $C$.
    Hence, for each $v \in S$ and $t \in C$, it holds that $\lambda_{v,t} = \lambda_t$ for some~$\lambda_t \in \R$.

We next prove that the coefficients $\lambda^{}_t=\lambda_{\hat t}$ for all $t, t' \in C$ if $|C|\ge 2$.
    By~\ref{perfectness:separating-2}, there exists $z \in Z$ with $|\mathcal{K}_z|\geq |S|-|C|+2$.
    Let us fix $W \in \mathcal{K}_z$, and define $\mathcal{K}'_z \subseteq \mathcal{K}_z \setminus \{W\}$ with $|\mathcal{K}'_z|=|S|-|C|+1$, and let $\mathcal{K}'' = \mathcal{K}'\setminus (\mathcal{K}'_z \cup \{W\})$.
    Note that $|\mathcal{K}''| = |S| - (|S|-|C|+2) = |C|-2$.
    Consider classes $t, t' \in C$ such that~$t \neq t'$, and any bijection $\pi \colon \mathcal{K}'' \to C \setminus\{t, t'\}$.
    Since $\bigcup_{K \in \mathcal{K}'_z} V(K) \cup V(W) \cup \{z\}$ induces a connected graph, the vectors $e'+e(W, t)$ and $e' +e(W,  t')$ belong to $\hat F$, where 
    \[e' = e(z,t) + \sum_{K \in \mathcal{K}'_z} e(K,t) + \sum_{K \in \mathcal{K}''} e(K,\pi(K)).\]
    Consequently, it holds that $\lambda^{}_{t}=\lambda_{ t'} = \lambda_{*}$ for all $t,  t' \in C$.

Consider now a vertex $z  \in Z$,  and a class $t \in C$. 
    Let $\mathcal{K}'_z \subseteq \mathcal{K}^{}_z$ with $|\mathcal{K}'_z | = |S|-|C|+1$, and 
    let $\mathcal{K}'' = \mathcal{K'}\setminus \mathcal{K}'_z$.
Note that $|\mathcal{K}''|=|C|-1$, and consider any bijection  $\pi \colon \mathcal{K}'' \to C\setminus \{t\}$. 
    We define the vector 
    \[ e' = e(z,t) + \sum_{K \in \mathcal{K}'_z} e(K,t) + \sum_{K \in \mathcal{K}''} e(K, \pi(K)).\]
    Since $\left (\bigcup_{K \in \mathcal{K}'_z} V(K)\right) \cup \{z\}$ induces a connected subgraph of $G$,  $e'$ belongs to~$\hat F$.
Observe now that, for each~$K \in \mathcal{K}'_z$, the vertices of $K$ and the vertices of the components in  $\mathcal{K}''$ induce a connected subpartition of $G$ into $|C|$ classes. So  we have
    \[\lambda^T \left (e(K,t) + \sum_{K \in \mathcal{K}''} e(K, \pi(K)) \right) = \lambda_0,\] 
    and since $\lambda^T  e(K,t)=\lambda_*$, we have $\lambda^T \left ( \sum_{K \in \mathcal{K}''} e(K, \pi(K)) \right) = \lambda_0 - \lambda_*$. 
       As a consequence, it holds that
    \[\lambda^T e' = \lambda_{z,t} + (|\mathcal{K}'_z|-1) \lambda_* + \lambda_0  = \lambda_0.\]
    Therefore, we have $\lambda_{z,t} = -(|S|-|C|)\lambda_*$ for all $z \in Z$ and $t \in C$.
    It follows that $\lambda = \lambda_* \hat \lambda$, and thus inequality~\eqref{ineq:multicolor-cut} induces a facet of $\poly(G,k)$.
 \end{proof}

\subsection{Pairing inequalities}\label{subsec:pair}

We next introduce inequalities that also generalize~\eqref{ineq:cut} to multiple classes.
Differently from~\eqref{ineq:multicolor-cut}, each class in these inequalities is associated with a pair of non-adjacent vertices, and assigning one pair of vertices to a certain class may prevent other vertices from being assigned to another class.
Before presenting a formal definition, let us briefly give some intuition using the following simple example.

Let~$k \in \Z$ with $k\ge 2$, let $a,b \in [k]$ with $a\neq b$, and let $P=v_1\ldots v_5$ be a path.
First note that $x_{v_1,a}+x_{v_2,b}+x_{v_4,b}+x_{v_5,a} - x_{v_3,a} - x_{v_3,b} \leq 2$ is a valid inequality for $\poly(P,k)$ since it is simply the sum of inequality~\eqref{ineq:cut} for vertices $\{v_1,v_5\}$ in class $a$, $\{v_2,v_4\}$ in class $b$, and $\{v_3\}$ as a separator of both pairs.
Moreover, it is tight as one could assign $v_1$ to class $a$, $v_2$ to class $b$, and leave the remaining vertices unassigned.
The inequality $x_{v_1,a}+x_{v_2,b}+x_{v_4,b}+x_{v_5,a} - x_{v_3,b} \leq 2$ dominates the previous one, and it is also valid for  $\poly(P,k)$.
The validity of the last inequality follows from the observation that $v_2$ and $v_4$ cannot be assigned to class $b$ if  $v_1$ and $v_5$ are assigned to class~$a$.

We next define inequalities based on the idea that assigning some vertices to a certain class may prevent other vertices from being assigned to another class, as illustrated in the previous example.

Let $C \subseteq [k]$ be a set of classes.
For each $c \in C$, let $u_c, v_c \in V(G)$ such that $u_c\neq v_c$ and $\{u_c,v_c\} \notin E(G)$.
The vertices $u_c$ and $v_c$ are called \emph{delegates} of class~$c$.

Consider now a set of vertices $Z \subseteq V(G)\setminus (\bigcup_{c \in C} \{u_c,v_c\})$ such that $Z$ is a $u_c,v_c$-separator for every $c \in C$.
Given $c \in C$, the pivot  vertex of a simple path $P$ between $u_c$ and~$v_c$ (i.e., a $u_c,v_c$-path) is an arbitrary vertex in $V(P)\cap Z$.
Let us fix a set of pivots in $Z$ for every~$u_c,v_c$-path with $c \in C$.
Given $z \in Z$, $c \in C$, and $b \in C\setminus \{c\}$, we say that $c$ \emph{blocks} $b$ \emph{via} $z$ if every  $u_c,v_c$-path in $G$ containing $z$ 
as its pivot in $Z$
also contains $u_b$ and $v_b$.
For every $z \in Z$ and $c \in C$, we define~$\gamma_{z,c} = 0$ if $c$ blocks $b$ via $z$ for some $b \in C\setminus \{c\}$, or if there is no $u_c,v_c$-path in~$G$ containing~$z$ as its 
pivot in $Z$. 
Otherwise, we define $\gamma_{z,c} = 1$. 

In the example above, $C=\{a,b\}$, $u_a=v_1$, $v_a=v_5$, $u_b=v_2$, $v_b=v_4$, $Z=\{v_3\}$ and class $a$ blocks class $b$ via $z=v_3$ since the unique path connecting $v_1$ and $v_5$ and containing $v_3$ also contains $v_2$ and $v_4$. Consequently, $\gamma_{v_3,a}=0$ and $\gamma_{v_3,b}=1$. Now suppose that we add a new vertex $v_6$ and two edges $\{v_1,v_6\}$  and $\{v_5,v_6\}$ converting the path to a cycle. With the same choice of classes and delegates, the set $Z=\{v_3,v_6\}$ is a separator for both pairs $v_1,v_5$ and $v_2,v_4$. In this setting, $a$ blocks $b$ via $v_3$ and $b$ blocks $a$ via $v_6$.
Hence the following inequality holds for this example: $x_{v_1,a}+x_{v_2,b}+x_{v_4,b}+x_{v_5,a} - x_{v_3,b} - x_{v_6,a} \leq 2$.

\begin{proposition}
Let $C \subseteq [k]$ with $|C|\ge 1$, and let $\{\{u_c, v_c\}\}_{c\in C}$ be a collection of pairs of vertices in $V(G)$ such that $u_c\neq v_c$ and $\{u_c,v_c\} \notin E(G)$ for all $c \in C$.
Let $Z \subseteq V(G)\setminus (\bigcup_{c \in C} \{u_c,v_c\})$ such that $Z$ is a separator of $u_c$ and $v_c$, for every $c \in C$, and consider a set of pivots in $Z$ for every $u_c,v_c$-path with $c \in C$. 
Finally, let $\gamma$ be defined accordingly.
The following inequality is valid for $\poly(G,k)$.
    \begin{equation}\label{ineq:blocking}
        \sum_{c\in C} \left (x_{v_c,c} + x_{u_c,c} - \sum_{z\in Z} \gamma_{z,c} \:x_{z,c} \right) \leq |C|.
    \end{equation}
\end{proposition}
\begin{proof}
    We assume that $|C|\ge 2$, otherwise $C = \{c\}$ for some $c \in [k]$, and so~\eqref{ineq:blocking} is precisely inequality~\eqref{ineq:cut}.
Suppose first that, for every $c \in C$ and $u_c,v_c$-path $P$, it holds that $\gamma_{z,c}=1$ for the 
    pivot vertex $z$ of $P$.
In this case, \eqref{ineq:blocking} is trivially valid for $\poly(G,k)$ as each class in $C$ contributes with at most $1$ to the left-hand side of the inequality.

    Consider now a vector $\hat x \in \poly(G,k) \cap \{0,1\}^{nk}$, and let $\{V'_i\}_{i \in [k]}$ be the connected $k$-subpartition of $G$ corresponding to $\hat x$.
    We define a $k$-subpartition $\mathcal{W}= \{V_i\}_{i \in [k]}$ of $G$ satisfying, for each $c \in C$:
    \begin{itemize}
        \item $V_c = V'_c\cap \{u_c,v_c\}$ if $|V'_c\cap\{u_c,v_c\}| \leq 1$, and
        \item $V_c$ is the set of vertices of any $u_c,v_c$-path in $G[V'_c]$ otherwise.
    \end{itemize}
    Moreover, define $V_i=\emptyset$ for all $i \in [k]\setminus C$.
    It is clear from this definition that $\mathcal{W}$ is a connected $k$-subpartition of $G$, and that the value of the left-hand side of~\eqref{ineq:blocking} given by $\bar x:= \chi(\mathcal{W})$ is at least $\sum_{c\in C} \left (\hat x_{v_c,c} + \hat x_{u_c,c} - \sum_{z\in Z} \gamma_{z,c} \:\hat x_{z,c} \right)$, that is, the same given by $\hat x$.

    Suppose there exist $z \in Z$ and~$a \in C$ such that $\gamma_{z, a} = 0$, $V_{a} \supseteq \{u_{a},v_{a},{z}\}$, 
    and $z$ is the pivot of $G[V_a]$.
    It follows from the definition of~$\bar x$ and~$\gamma$ that $V_{a} \supseteq \{u_b,v_b\}$ for some $b \in C\setminus \{a\}$, and so $V_b \cap \{u_b, v_b\} =\emptyset $.
    Since the classes in $\mathcal{W}$ are pairwise disjoint, and each non-empty class induces a path, 
    we conclude that the cardinality of 
$\{c \in C  : V_{c} \supseteq\{u_c,v_c\}, \text{ and } \gamma_{z,c}=0 \text{ for all } z\in Z\cap V_c\}$ 
    is at most the size\textbf{} of $\{c \in C : V_c \cap \{u_c,v_c\} = \emptyset\}$.
    Thus $\bar x$ satisfies~\eqref{ineq:blocking}, so does~$\hat x$.
    Therefore, the inequality is valid for $\poly(G,k)$.
 \end{proof}

Figure~\ref{subfig:blocking:1} shows a graph where $\{v_1,v_2\}$ is a separator, $\{v_3, v_5\}$ and $\{v_4,v_5\}$ are the  delegates of classes $a$ (green)  and $b$ (yellow), respectively, with $u_a = v_3$ and $u_b=v_4$.
Since every (simple) path between delegates of same class contains exactly one vertex of $Z$, the choice of the pivots is trivial. 
The corresponding inequality~\eqref{ineq:blocking} is $x_{v_3,a}+x_{v_5,a}+x_{v_4,b}+x_{v_5,b}-x_{v_1,b}-x_{v_2,b} \leq 2$. 
Figure~\ref{subfig:blocking:2} illustrates the same graph with $\{v_1,v_4\}$ as the separator, $\{v_3,v_5\}$ and $\{v_2,v_3\}$ are the  delegates of classes $a$  and $b$, respectively, with $u_a=u_b=v_3$.
Consider first $v_1$ as the pivot of the paths $v_3v_4v_1v_5$ and $v_3v_4v_1v_5v_2$, and $v_4$ as the pivot of $v_3v_4v_2v_5$ and  $v_3v_4v_2$.
In this case, the corresponding inequality~\eqref{ineq:blocking} is $x_{v_3,a}+x_{v_5,a}+x_{v_2,b}+x_{v_3,b}-x_{v_1,a}-x_{v_4,b} \leq 2$. 
If one chooses $v_4$ as the pivot of every path linking delegates of  the same class, then the obtained inequality is 
$x_{v_3,a}+x_{v_5,a}+x_{v_2,b}+x_{v_3,b}-x_{v_4,a}-x_{v_4,b} \leq 2$, which is implied by~\eqref{ineq:cut} as $v_4$ is a cut-vertex.

As we shall prove later in Theorem~\ref{thm:blocking:facetness}, 
the inequalities in Figure~\ref{fig:example:blocking} are facet-defining.
Notice that, in contrast to the case $|C|=1$ (i.e., inequalities~\eqref{ineq:cut}), the minimality of the separator is no longer a necessary condition for facetness when multiple classes ($|C|\ge2$) are considered.

\begin{figure}[t!]
\centering
\begin{subfigure}{.47\linewidth}
    \centering
\begin{tikzpicture}[
auto,
            node distance = 1.5cm, semithick, scale=0.8, 
]
\tikzstyle{every state}=[
            draw = black,
            semithick,
minimum size = 4mm,
            scale=0.9,
            inner sep = 1.5pt
        ]

\node[state] (v4) [fill=yellow] {$v_4$};
        \node[state] (v1) [fill=yellow,above right of=v4] {$v_1$};
        \node[state] (v2) [fill=yellow,below right of=v4] {$v_2$};
        \node[state] (v3) [below left of = v4, fill=green] {$v_3$};
\node[state] (v5) [below right of=v1, semifill={upper=green, lower=yellow, ang=90}] {$v_5$};

        \node[draw,dotted,fit=(v1) (v2)] {};

        \path (v4) edge node{} (v1);
        \path (v4) edge node{} (v2);
        \path (v4) edge node{} (v3);
        \path (v5) edge node{} (v1);
        \path (v5) edge node{} (v2);
        
    \end{tikzpicture}
    \caption{$x_{v_3,a}+x_{v_5,a}+x_{v_4,b}+x_{v_5,b}-x_{v_1,b}-x_{v_2,b} \leq 2$
}
    \label{subfig:blocking:1}

\end{subfigure}\hfill
\begin{subfigure}{.47\linewidth}
    \centering
    \begin{tikzpicture}[
auto,
            node distance = 1.5cm, semithick, scale=0.8, 
]
\tikzstyle{every state}=[
            draw = black,
            thick,
minimum size = 4mm,
            scale=0.9,
            inner sep=1.5pt
        ]

\node[state] (v4) [fill=yellow] {$v_4$};
        \node[state] (v1) [fill=green,above right of=v4] {$v_1$};
        \node[state] (v2) [fill=yellow,below right of=v4] {$v_2$};
        \node[state] (v3) [below left of = v4, semifill={upper=green, lower=yellow, ang=90}] {$v_3$};
\node[state] (v5) [below right of=v1, fill=green] {$v_5$};

        \node[draw, rotate fit=45, minimum size=8mm,  inner sep= 2mm, dotted, fit=(v1) (v4)] {};

        \path (v4) edge node{} (v1);
        \path (v4) edge node{} (v2);
        \path (v4) edge node{} (v3);
        \path (v5) edge node{} (v1);
        \path (v5) edge node{} (v2);
        
    \end{tikzpicture}
    \caption{$x_{v_3,a}+x_{v_5,a}+x_{v_2,b}+x_{v_3,b}-x_{v_1,a}-x_{v_4,b} \leq 2$
}
    \label{subfig:blocking:2}
\end{subfigure}\hfill
    \caption{
    Examples of inequality~\eqref{ineq:blocking}.
Dotted rectangles represent the separators, minimal in \eqref{subfig:blocking:1} and not minimal in \eqref{subfig:blocking:2}. Green and yellow vertices represent the vertices of classes~$a$ and $b$, respectively, such that the corresponding variables have non-zero coefficients in~\eqref{ineq:blocking}.}
    \label{fig:example:blocking}
\end{figure}

We next define a class of connected $k$-subpartitions of $G$ whose corresponding vectors in $\poly(G,k)$ satisfy~\eqref{ineq:blocking} with equality.
A connected $k$-subpartition $\{V_i\}_{i \in [k]}$ of $G$ is said to be \emph{conforming} with respect to $(C,\{u_c,v_c\}_{c \in C}, Z, \gamma)$ if, for each $c \in C$, one of the following holds:
\begin{enumerate}
    \item $|V_c\cap \{u_c, v_c\}|=0$, $\gamma_{z,c}=0$ for all $z \in V_c \cap Z$, and there is $b\in C\setminus\{c\}$ such that $V_b \supseteq \{u_c,v_c,u_b,v_b\}$; or
    \item $|V_c\cap \{u_c, v_c\}|=1$, $\gamma_{z,c}=0$ for all $z \in V_c \cap Z$, and $V_c \not\supseteq \{u_b,v_b\} $ for all $b \in C\setminus\{c\}$; or
    \item $|V_c\cap \{u_c, v_c\}|=2$, $\gamma_{z,c}=1$ for exactly one $z \in V_c \cap Z$, and $V_c \not\supseteq \{u_b,v_b\} $ for all $b \in C\setminus\{c\}$; or
    \item $|V_c\cap \{u_c, v_c\}|=2$, $\gamma_{z,c}=0$ for all $z \in V_c \cap Z$, and $V_c \supseteq \{u_b,v_b\} $ for exactly one $b \in C\setminus\{c\}$.
\end{enumerate}

\begin{lemma}\label{lemma:conforming-equality}
    Let $\mathcal{V}=\{V_i\}_{i\in [k]}$ be a conforming connected $k$-subpartition of $G$.
    It holds that the incidence vector $\chi(\mathcal{V})$ satisfies~\eqref{ineq:blocking} with equality.
\end{lemma}
\begin{proof}
    Consider a class $c \in C$.
    We say that $c$ is \emph{type}-$i$ with  $i \in  \{1,2,3,4\}$ if it satisfies case~$i$ in the definition of conforming.
    If~$c$ is type-2 or type-3, then~$c$ contributes with one to the left-hand side of inequality~\eqref{ineq:blocking}, and does not prevent any other class in $C$ from contributing to it.
    If~$c$ is type-4, then there is precisely one class $b \in C\setminus\{c\}$ such that $V_c \supseteq\{u_b,v_b\}$.
    In this case, classes $b$ and~$c$ together contribute  with exactly two to the left-hand side of~\eqref{ineq:blocking}.
    If~$c$ is type-1, then there is $b \in C\setminus\{c\}$ such that $V_b\supseteq\{u_c,v_c,u_b,v_b\}$, and so~$b$ is type-4.
    Again, $b$ and  $c$ contribute together with two to the left-hand side.
    Therefore, $\chi(\mathcal{V})$ satisfies inequality~\eqref{ineq:blocking} with equality.
 \end{proof}

If $|C|=1$, we recall that \eqref{ineq:blocking} is exactly~\eqref{ineq:cut} (where $c$ is the single class in $C$), and that this inequality defines a facet of $\poly(G,k)$ if and only if the separator is minimal, as shown in Corollary~\ref{cor:simple-facets}.
In the next theorem, we prove that~\eqref{ineq:blocking} is facet-defining when $|C|\ge2$ under some hypothesis on the existence of certain conforming connected subpartitions.
For every $c,b \in [k]$, we say that  $c$ \emph{blocks} $b$ if there exists $z \in Z$ such that $c$ blocks $b$ via $z$.

\begin{theorem}\label{thm:blocking:facetness}
    Inequality~\eqref{ineq:blocking} with $|C|\geq 2$ induces a facet of $\poly(G,k)$ if the following conditions hold:
    \begin{enumerate}[(i)]
\item For each $c \in [k]$ and $v \in V(G)$, there is a conforming connected $k$-subpartition $\{V_i\}_{i \in [k]}$ of $G$ such that $v \in V_c$.
         Moreover, if $v$ is a delegate of $c$, then $\bar v \notin V_i$ for all $i\in [k]$, where $\bar v$ is a delegate of $c$ different from~$v$.  \label{property:conforming2}
\item For each $c,b \in C$ such that $c$ blocks $b$ 
there is a conforming connected $k$-subpartition such that $V_c \supseteq \{u_b,v_b,u_c,v_c\}$. \label{property:blocking-class}
        \item The class-blocking  graph $B$ is connected, where $V(B)=C$ and $E(B)=\{\{a,b\} : a \text{ blocks } b\}$. \label{property:blocking-graph}

    \end{enumerate}
\end{theorem}
\begin{proof}

     Let $\hat{F} = \{ x\in \poly(G,k) : \hat{\lambda}^{T}x = \hat{\lambda}_0 \}$ be a face of $\poly(G,k)$ where $(\hat{\lambda}^{T}, \hat{\lambda}_0)$ corresponds to inequality~\eqref{ineq:blocking}.
    Suppose that $F= \{ x\in \poly(G,k) : \lambda^T x = \lambda_0 \}$ is a face of $\poly(G,k)$ which contains~$\hat{F}$.

    Throughout this proof, we implicitly use Lemma~\ref{lemma:conforming-equality}  to show vectors in $\hat F$ by constructing conforming connected subpartitions.
    Let us first show the zero entries of $\lambda$.
    Consider $c \in [k]$ and $v \in V(G)$.
    By~\ref{property:conforming2}, there is a conforming connected $k$-subpartition $\mathcal{V} = \{V_i\}_{i \in [k]}$ of $G$ such that~$v \in V_c$.
    Suppose  first that $c \in [k]\setminus C$.
    We may assume that $V_c=\{v\}$, otherwise we may simply remove all vertices from $V_c$ but~$v$. 
    Since $\mathcal{V}$ is conforming, the vectors 
\(\chi(\mathcal{V})\) and \(\chi(\mathcal{V})- e(v,c)\) belong to~$\hat F$.
    This implies $\lambda_{v,c}=0$ for every $v \in V(G)$, and~$c \in [k]\setminus C$.
    
    Suppose that $c \in C$ and  $v \in V(G)\setminus (D \cup Z)$, where $D = \bigcup_{c\in C}\{u_c,v_c\}$. 
If $|V_c\cap \{u_c,v_c\}|=0$, then assume with no loss of generality that $V_c=\{v\}$.
If $|V_c\cap \{u_c,v_c\}|=1$, then assume (w.l.o.g.) that $V_c$ induces a path with one endpoint at $v$ and the other at one of the delegates $u_c$ or $v_c$.
In both cases,  removing $v$ from $G[V_c]$ does not increase the number of components, and consequently
\(\chi(\mathcal{V})\) and \(\chi(\mathcal{V})- e(v,c)\) belong to~$\hat F$.
    Suppose now that $|V_c\cap \{u_c,v_c\}|=2$, and assume (w.l.o.g.) that $V_c$ induces a minimal connected subgraph containing $u_c,v_c$, and~$v$. 
If $v$ is not a cut-vertex of $G[V_c]$, then
\(\chi(\mathcal{V})\) and \(\chi(\mathcal{V})- e(v,c)\) belong to~$\hat F$.
    Otherwise, every path between $u_c$ and $v_c$ in $G[V_c]$ contains $v$.
If~$\gamma_{z,c} = 0$ for all $z \in V_c \cap Z$, then let $\hat V_c$ be the set of vertices of a path in~$G[V_c]$ between $v$ and either $u_c$ or $v_c$.
    Otherwise, there is exactly one $z \in V_c \cap Z$ such that $\gamma_{z,c}=1$, and thus we define $\hat V_c$ as the set of vertices of a path in~$G[V_c]$ between $v$ and either $u_c$ or $v_c$, which does not contain $z$.
In the later case, the vectors 
    \[\sum_{i \in [k]\setminus\{c\}} e(V_i,i) + e(\hat V_c,c) \quad \text{ and } \quad \sum_{i \in [k]\setminus\{c\}} e(V_i,i) + e(\hat V_c,c) - e(v,c)\] 
    belong to $\hat F$.
    In the former case, there exists $b \in C\setminus \{c\}$ such that $V_c \supseteq \{u_b,v_b\}$. 
    Let us assume (w.l.o.g.) that $u_b \notin \hat V_c$.
    Additionally, we assume $V_b=\emptyset$ as one may remove all vertices from $V_b$ since $V_b\cap \{u_b,v_b\} = \emptyset$.
    Hence the  vectors $e'$ and $e' - e(v,c)$ belong to $\hat F$, 
where \[e' = \sum_{i \in [k]\setminus\{b,c\}} e(V_i,i) + e(\hat V_c,c) + e(u_b,b).\]
    The vectors described above show that $\lambda_{v,c}=0$ for all $v \in V(G)\setminus (D \cup Z)$, and~$c \in [k]$.
    Using very similar arguments, one may show the remaining zero entries of $\lambda$, that is, for every $c \in C$,  $v \in Z$ with $\gamma_{v,c}=0$,  or $v \in D$ with~$v \notin \{u_c,v_c\}$.

We next show the non-zero coefficients of $\lambda$.
    Consider $c \in C$, and a vertex~$v \in V(G)$ such that $v$ is a delegate of class $c$.
    Assume that $v=v_c$ and let $\{V_i\}_{i \in [k]}$ be a conforming connected $k$-subpartition such that $v_c \in V_c$ and $u_c \notin V_i$ for all $i \in [k]$.
    The existence of such a conforming subpartition is guaranteed by~\ref{property:conforming2}.
    Note that $c$ is type-2 in this subpartition, and so we can remove all vertices from $V_c$ but~$v_c$ and still obtain a conforming connected $k$-subpartition.
    It follows from this observation that the following vectors are in $\hat F$:
    \[\sum_{i \in [k]\setminus \{c\}} e(V_i, i) + e(v_c,c) \quad \text{ and } \quad \sum_{i \in [k]\setminus \{c\}} e(V_i, i) + e(u_c,c).\]
    This implies that $\lambda_{v_c,c} = \lambda_{u_c,c}$ for all $c \in C$.

    Let $c \in C$, and let $z \in Z$ with $\gamma_{z,c}=1$.
    Consider a conforming connected $k$-subpartition $\{V_i\}_{i \in [k]}$ with $z \in V_c$, and note that $c$ must be type-3 
which implies $V_c\supseteq \{u_c,v_c,z\}$ and $\gamma_{\hat z,c}=0$ for all $\hat z \in (Z\cap V_c)\setminus \{z\}$.
    Thus we can remove all vertices of $V_c$ but one of the delegates, say $u_c$, and still have a conforming connected $k$-subpartition.
    Consequently, it holds that 
    \(\sum_{i \in [k]} e(V_i,i)\) and $\sum_{i \in [k]\setminus\{c\}} e(V_i,i) + e(u_c,c)$ belong to $\hat F$.
    It follows that $-\lambda_{z,c}=\lambda_{v_c,c}=\lambda_{u_c,c}$ for all $c \in C$, and $z \in Z$ with $\gamma_{z,c}=1$.

    Consider now $b,c \in C$ such that $c$ blocks $b$.
    By~\ref{property:blocking-class}, there is a conforming connected $k$-subpartition $\mathcal{V} = \{V_i\}_{i\in [k]}$ such that $V_c\supseteq \{u_b,v_b,u_c,v_c\}$.
    Since $u_c\neq v_c$ and $u_b\neq v_b$, we assume (w.l.o.g.) that $u_b \neq v_c$, and that $V_b=\emptyset$.
    Note that one may remove the class $V_c$ from~$\mathcal{V}$, add classes $\{u_b\}$ and $\{v_c\}$, and obtain a conforming connected $k$-subpartition.
    As a consequence, the vectors $\sum_{i \in [k]} e(V_i,i)$ and $\sum_{i \in [k]\setminus\{c\}} e(V_i,i) + e(u_b,b) + e(v_c,c)$  are in $\hat F$, and so~$\lambda_{u_b,b}=\lambda_{v_c,c}$.
    By~\ref{property:blocking-graph}, the class-blocking graph is connected, and so we have $\lambda_{u_b,b}=\lambda_{v_b,b}=\lambda_{u_c,c}=\lambda_{v_c,c}$ for every~$b,c \in C$.
    Therefore, there is $\lambda_* \in \R$ such that $\lambda=\lambda_* \hat \lambda$, and thus inequality~\eqref{ineq:blocking} induces a facet of $\poly(G,k)$.
 \end{proof}

To conclude this section, we observe that condition~\ref{property:blocking-graph} in the previous theorem is also necessary for facetness.
Indeed, inequality~\eqref{ineq:blocking} given by the tuple $(G,C,\{u_c,v_c\}_{c \in C}, Z,\gamma)$ is the sum of~\eqref{ineq:blocking} given by $(G,C',\{u_c,v_c\}_{c \in C'}, Z,\gamma)$ for each $C'\subseteq C$ such that $C'$ is the set of vertices of a connected component of the class-blocking graph $B$.

 \section{Separation of multiclass inequalities}
\label{sec:separation-multi}

We first show that the separation problem associated with inequalities~\eqref{ineq:multicolor-cut} (if~$S$ is given) is $\NP$-hard 
by  observing that it consists in solving the \textsc{Minimum Multiway Cut} problem (undirected node version) which is defined as follows.

Given a graph $G$, a set of terminals $S\subseteq V(G)$ inducing a stable set in $G$, and weights $w\colon V(G)\setminus S \to \R_\geq$, the problem consists in finding a minimum-weight set of non-terminal vertices $Z \subseteq V(G)\setminus S$ such that there is no path in $G-Z$ between any pair of distinct vertices in $S$, that is, each component of $G-Z$ contains at most one vertex in $S$.
Note that when $|S|=2$, this is exactly the \textsc{Minimum Vertex Cut} problem, and so it is solvable in polynomial time using a maximum flow algorithm.
For $|S|\geq 3$, Garg et al.~\cite{GARG200449} proved that  \textsc{Minimum Multiway Cut} is  $\NP$-hard.

Let us fix $k=1$. 
Given a graph $G$,  a stable set $S\subseteq V(G)$ of $G$ with $|S|\geq 3$, and a vector $\bar x \in \R_\ge^{n}$, the separation problem of inequalities~\eqref{ineq:multicolor-cut} on input $(G,S, \bar x)$ is equivalent to finding a minimum-weight set of vertices $Z^* \subseteq V(G)\setminus S$ such that each vertex of $S$ belongs to a different component of $G-Z^*$, where the weight of each vertex $v \in V(G)\setminus S$ is $\bar x_{v}$.
In other words,  solving the separation of~\eqref{ineq:multicolor-cut} is not easier than   solving \textsc{Minimum Multiway Cut}.
Therefore the separation problem of inequalities~\eqref{ineq:multicolor-cut}, when $S$ is given, is $\NP$-hard  even for $k=1$.

We next describe a simple heuristic to find violated inequalities~\eqref{ineq:multicolor-cut} using separating sets.
A set of vertices $Z \subset V(G)$ is a \emph{separating set} (or \emph{vertex cut}) of a connected graph $G$ if $G-Z$ has more than one component.
Given $k \in \Z_\ge$, $C \subseteq [k]$, a connected graph $G$, and  $\bar x \in \R_\ge^{nk}$, 
the weight of each vertex $v \in V(G)$ is defined as $f(v)=\sum_{c \in C} \bar x^*_{v,c}$.
A natural heuristic approach to find inequalities~\eqref{ineq:multicolor-cut} violated by~$\bar x$ is to first compute a minimum separating set~$Z$ of $G$ with weight function~$f$, and then choose a vertex of maximum weight in each component of $G-Z$ to be part of $S$. 
The problem of finding a minimum separating set can be solved in polynomial time by reducing it to the \textsc{Maximum Stable Set} problem on weighted bipartite graphs as shown by Balas and Souza~\cite{balas2005vertex}. 
Regarding the subset $C\subseteq [k]$, it is not clear to us what is the best choice for this set of classes. 
A simple strategy could be to consider the subsets of $[k]$ in nondecreasing order of cardinality.
Of course, this enumeration approach is viable only when $k$ is small.

In the remainder of this section, we focus on the following separation problem associated with inequalities~\eqref{ineq:blocking} when $C$ and the delegates are already chosen.
Given $k \in \Z_\ge$, $C \subseteq [k]$ with $|C|\ge 1$, a graph~$G$, a collection of pairs of vertices $\{\{u_c, v_c\}\}_{c \in C}$ with $u_c\neq v_c$ such that $\{u_c,v_c\} \notin E(G)$, and $\bar x \in \R^{nk}_\ge$, the separation problem of~\eqref{ineq:blocking}  consists in finding a multicut $Z$ of the pairs $\{\{u_c,v_c\}\}_{c \in C}$ in $G$ that minimizes 
\[ \sum_{v \in Z} \sum_{c \in C}\gamma_{v,c} \: \bar x_{v,c}.\]

At first glance, one might think that this separation problem implies solving the \textsc{Minimum Multicut} problem  defined as follows.
Given a graph $G$ and a collection $Q$ of pairs of distinct non-adjacent vertices in $G$, the goal is to find a minimum-size \emph{multicut} $Z$ of $Q$ in $G$, that is, a set of vertices~$Z \subseteq V(G) \setminus (\bigcup_{q \in Q} q)$ such that, for each pair $\{u,v\} \in Q$, there is no path in $G-Z$ between $u$ and $v$.
C\u{alinescu et al.~\cite{CALINESCU2003333}} 
proved that \textsc{Minimum Multicut} is $\NP$-hard on trees with maximum degree at most~$4$.
However, we next prove that the separation problem of~\eqref{ineq:blocking} is polynomial-time solvable on trees by reducing it to \textsc{Minimum-Cost Flow}.

\begin{theorem}\label{thm:blocking:separation-trees}
Let $k \in \Z_\ge$, let $C \subseteq [k]$ with $|C|\ge 1$, let~$G$ be an $n$-vertex tree, and let~$Q=\{\{u_c, v_c\}\}_{c \in C}$ be a collection of pairs of vertices  with $u_c\neq v_c$ such that $\{u_c,v_c\} \notin E(G)$.
The separation problem of~\eqref{ineq:blocking} can be solved in $\bigO(n^3 \log n \log\log |C|)$ time. 
\end{theorem}
\begin{proof}
For each $c \in C$, let us denote by $P_c$ the path in $G$ between $u_c$ and~$v_c$.
     Note that if there exist distinct classes $b,c \in C$ such that $V(P_c)\subseteq V(P_b)$, then we can simply remove the pair $\{u_b,v_b\}$ from $Q$ since every separator~$Z$ of $u_c$ and $v_c$ is also a separator of $u_b$ and $v_b$, and $\gamma_{v,b} = 0$ for all $v \in V(G)$ as $b$ blocks $c$. Therefore, we assume from now on that $V(P_c)\not\subseteq V(P_b)$ for all $b,c \in C$ with $b\neq c$.
    Thus, for each $c \in C$ and $u_c,v_c$-separator $Z$, there is a single vertex $v \in (V(P_c)\setminus \{u_c,v_c\}) \cap Z$ such that $\gamma_{v,c}=1$.

Let us define $V' = \bigcup_{q \in Q} q$. We next construct a network consisting of a directed graph $D$ with vertices $V(D) = (V(G)\setminus V') \cup C \cup \{s,t\}$ and arcs 
\(A(D) = \{(s,v) : v \in V(G)\setminus V'  \} \:\cup \{(c,t) :  c \in C  \} \cup A',\)
where \(A' = \{(v,c) \in  (V(G)\setminus V') \times C :  v \in V(P_c)\}.\)
We define the demands in the network as $d \colon V(D) \to \Z_\ge$ such that $d(s) = -|C|$, $d(t) = |C|$, and $d(v) = 0$ for all $v \in V(D)\setminus \{s,t\}$.
Moreover, the capacity function $r \colon A(D)\to \Z_\ge$ is defined as 

\[
 r(a) = 
  \begin{cases} 
   |C|       & \text{if } a = (s,v) \text{ with } v \in V(G)\setminus V',\\
   1 & \text{otherwise.}
  \end{cases}
\]
Finally, the cost of the arcs in the network is given by a function $w \colon A(D) \to \Z_\ge$ such that $w(a) = \bar x_{v,c}$ if $a = (v,c) \in A'$, and $w(a) = 0$ otherwise.

Since $d$ and $r$ are integer-valued functions, a minimum-cost flow~$f^*$  in the  network $(D,d,r,w)$ such that $f^*(a) \in \Z$ for all $a \in A(D)$ can be computed in $\bigO(n^3 \log n \log\log |C|)$ using the double scaling minimum-cost flow algorithm by 
Ahuja et al.~\cite{Ahuja1992} (see also Ahuja et al.~\cite{ahuja1993network}).
It follows from the definition of~$r$ that the flow of every arc in~$A'$  given by $f^*$ is either zero or one.
Also the capacities guarantee that at most one unit of flow enters any vertex in $C$.
We define $Z = \{v \in V(G)\setminus V' : f^*(v,c) = 1 \text{ for some } c \in C\}$, and observe that~$Z$ is a $u_c,v_c$-separator in $G$ for every $c \in C$ as the demand $d(t)=-d(s)=|C|$.
Moreover, for all $v \in Z$ and $c \in C$, we define $\gamma_{v,c} = f^*(v,c)$.
The cost of $f^*$ is precisely
\[\sum_{a \in A(D)} w(a) f^{*}(a) = \sum_{a \in A'} w(a) f^{*}(a)  = \sum_{v \in V(G)\setminus V'} \sum_{c \in C} \bar x_{v,c} f^*(v,c) = \sum_{v \in Z} \sum_{c \in C}\gamma_{v,c} \bar x_{v,c}. \]

On the other hand, given a set  $Z  \subseteq V(G) \setminus V'$ which is a $u_c,v_c$-separator for every $c \in C$, it induces a (feasible) flow $f$ in $(D,d,r,w)$ where $f(v,c)=\gamma_{v,c}$ for all $v \in V(G)\setminus V'$ and $c \in C$, where $\gamma_{v,c}$ is defined with respect to $Z$.
Therefore, the separation problem of~\eqref{ineq:blocking} can be solved in the same running time of the algorithm for \textsc{Minimum-Cost Flow} when the input graph $G$ is a tree.
 \end{proof}

We conclude this section observing that the algorithm described in the proof of Theorem~\ref{thm:blocking:separation-trees} does not seem to be adaptable to more general graphs.
Indeed, we do not know the computational complexity of the separation problem of~\eqref{ineq:blocking} in general.

 \section{Computational experiments\label{sec:experiments}}

\noindent In this section we investigate the use of the proposed single-class and multiclass inequalities for solving random graph instances generated in the same way as the ones described by Wang et al.~\cite{wang2017imposing} for the Maximum-Weight Connected Subgraph problem, as well as for computationally more challenging instances.
In our experiments, we address the following natural generalization of this classic problem.

\medskip

\noindent\begin{mws-problem}\hfill\\
\textsc{Input:} A graph $G$, a vertex-weight function $w\colon V(G) \to \Z$ (possibly with negative weights), and $k \in \Z_\geq$. \\
\textsc{Output:} A subgraph $H$ of $G$ with at most $k$ connected components.\\ 
\textsc{Objective:} Maximize $w(H) := \sum_{v \in V(H)} w(v)$.
\end{mws-problem}

Wang et al.~\cite{wang2017imposing} reported computational experiments on instances consisting of graphs with 50 vertices that are generated using the Erd\H{o}s–Rényi random graph model with probability $p \in \{0.01, 0.02, \ldots, 0.25\}$,  and  integer weights chosen uniformly at random in the interval $[-50, 50]$.
In our experiments, we consider a set of larger instances consisting of random graphs with 100 vertices
with probability $p \in \{0.01, 0.02, \ldots, 0.10\}$,  integer weights chosen uniformly at random in the interval $[-50, 50]$, and  $k \in \{5, 10, 15, 20, 25\}$.
We remark that every generated instance with $p > 0.10$ is solved by all algorithms within less than 0.1 second, and thus we do not report computational results for such instances.
For each combination of $p$ and $k$, we generate five instances.
In this way there are 250 instances in total of this type.

Inspired by the complete bipartite example presented in~\cite{wang2017imposing}, we have also generated a second instance set containing dense instances that are significantly more challenging than the previous ones.
This set is formed by bipartite graphs with 100 vertices (and both parts of same size) with edges appearing with probability $p \in \{0.05, 0.10, \ldots, 0.50\}$ and $k \in \{5, 10, 15, 20, 25\}$.  The weights of vertices in one part of the bipartition are chosen uniformly at random in the interval $[-50,0]$ and the vertices in the other part with weights in $[0,50]$.
Thus, we have 250 bipartite instances in total.

\subsection{Separation routines}\label{sec:exp:separation}

Let $(G,w,k)$ be an instance of \textsc{mws}, and let $\bar x \in \R^{nk}$ be a vector in a relaxation of $\poly(G,k)$. 
We first describe a simple heuristic procedure to generate cuts from the generalized connectivity inequalities~\eqref{ineq:cut-general}.
Consider a fixed class $c \in [k]$.
Intuitively, this procedure starts from an indegree inequality for $c$ and merges the singleton classes partitioning $V(G)$ (see Proposition~\ref{prop:indegree}) to obtain a coarser partition of $V(G)$ that induces a left-hand side larger than the left-hand side of the corresponding indegree inequality.

The algorithm starts with $S = V(G)$ and a vertex partition $\mathcal{W}=\{\{v\}\}_{v \in V(G)}$.
For every pair of distinct singleton classes $\{u\}$ and $\{v\}$ in $\mathcal{W}$, define
\(\Delta(u,v) = \sum_{z \in I} \bar x_{z,c}\) where $I= \{ z\in N(u)\cap N(v): \bar x_{z,c} \leq \min(\bar x_{u,c}, \:\bar x_{v,c}) \}$.
Suppose without loss of generality that $\bar x_{u,c} \geq \bar x_{v,c}$.
Define $\Delta'(u,v) = \Delta(u,v) - \bar x_{v,c}$ if $\{u,v\} \in E(G)$, and  $\Delta'(u,v) = \Delta(u,v)$ otherwise.
If $\Delta'(u,v) > 0 $, then let $\mathcal{W}'=(\mathcal{W}\setminus\{\{u\}, \{v\}\}) \cup \{\{u,v\}\}$ and let~$S'=S\setminus \{v\}$.
Finally, repeat the procedure above with $S'$ and $\mathcal{W}'$ playing the roles of $S$ and $\mathcal{W}$, respectively, until there are no singleton classes, or $\Delta'(\cdot, \cdot) \leq 0$ for all pairs of singleton classes.
Let~$\mathcal{W}$ be the partition of $V(G)$ produced by the previous procedure for~$\bar x \in \R^{nk}$ and  $c\in [k]$.
To obtain a generalized connectivity inequality~\eqref{ineq:cut-general}, we simply use Algorithm~\ref{alg:separation} on input $(\mathcal{W},c,\bar x)$ to compute the function $\hat d$.

The separation of multiway inequalities~\eqref{ineq:multicolor-cut} follows the general heuristic described in Section~\ref{sec:separation-multi}.
For each $v \in V(G)$, define $w'(v) = \sum_{c \in [k]} \bar x_{v,c}$.
Compute a minimum-weight vertex separator~$Z \subset V(G)$ of $(G,w')$ as proposed by Balas and Souza~\cite{balas2005vertex}.
For each component of $G-Z$, choose a vertex of maximum weight with respect to~$w'$.
These vertices form a stable set~$S$, which, together with~$Z$, defines a multiway inequality.

To separate the connectivity inequalities~\eqref{ineq:cut}, we proceed as  described by Miyazawa et al.~\cite{MIYAZAWA2021826}.
Fix a class $c \in [k]$, and construct a network $D$ with capacities $g\colon A(D) \to \Q_\ge \cup\{\infty\}$ as follows.
Define $V(D)=\{v_1, v_2 \colon v \in V(G)\}$ and $A(D) = A_1 \cup A_2$, where
$A_1 = \{(u_2,v_1), (v_2,u_1) \colon \{u,v\} \in E(G)\}$ and
$A_2 = \{(v_1, v_2) \colon v \in V(G)\}$.  
Finally, define $g(a) = \bar{x}_{v, c}$ if~$a = (v_1, v_2) \in A_2$; and $g(a) = \infty$ otherwise.  
Note that each arc in~$D$ with a finite capacity is
associated with a vertex of~$G$.  
For each pair of distinct vertices~$u,v \in V(G)$ with $\{u,v\} \notin E(G)$ with $\bar x_{u,c}+\bar x_{v,c} > 1$, compute in $D$ a
minimum (edge) cut between $u_1$ and $v_2$.  
If the weight of such a cut is
smaller than $\bar{x}_{u, c} + \bar{x}_{v, c} - 1$, then it is
finite and the vertices of~$G$ associated with the arcs in this cut
define a $(u,v)$-separator~$S$ in~$G$ that violates a connectivity
inequality for $u,v$ and $c$.

We use Goldberg's preflow algorithm implemented by LEMON Graph Lib.\ 1.3.1 to compute a minimum cut in $D$, and move to the next class when we find the first pair of vertices whose  corresponding connectivity inequality is violated.

\subsection{Computational environment and results}

All algorithms are implemented in C++ using LEMON Graph Lib.\ 1.3.1, and Gurobi 10 as the MILP solver with all Gurobi parameters set as default.
The experiments are carried out on a computer running Linux Ubuntu 22.04 LTS (64-bit) equipped with Intel Core i7-4790 at 3.6 GHz and 8GB of RAM\@.
The time limit for each instance is 900 seconds.

Let us denote by \textsc{bc} the 
branch-and-cut algorithm based on the following formulation for \textsc{mws}: $\max \sum_{v \in V(G)} \sum_{c \in [k]} w(v) x_{v,c}$ with $x \in \{0,1\}^{nk}$ satisfying~\eqref{ineq:cover} and~\eqref{ineq:cut}, where inequalities~\eqref{ineq:cut} are separated as described in the previous section.
The branch-and-cut algorithm \textsc{bc} with the exact separation of the indegree inequalities is denoted by \textsc{bc+i}.
We denote by \textsc{bc+g} the algorithm \textsc{bc} augmented with the generalized connectivity cuts generated with the heuristic described in  Section~\ref{sec:exp:separation}.
Finally, \textsc{bc+m} denotes the branch-and-cut algorithm including both the generalized connectivity inequalities and the multiway inequalities separated as discussed in Section~\ref{sec:exp:separation}.  

We use performance profiles introduced by Dolan and Mor{\'e}~\cite{dolan2002benchmarking} to compare the running times of the proposed algorithms. 
Let $A$ be a set of algorithms and $P$ a non-empty set of instances.
For each instance $p \in P$ and algorithm $a \in A$, $t_{p,a}$ denotes the running time required by $a$ to solve $p$.
The performance ratio of an algorithm~$a \in A$ on an instance~$p \in P$ is defined as~$r_{p,a} = t_{p,a}/\min\{t_{p,b} :  b \in A\}$.
In case $p$ is not solved by $a$, we set $r_{p,a}=+\infty$.
For every $a \in A$ and $\tau \geq 1$, define
\[\rho_{a}(\tau) = |\{p \in\textbf{} P : r_{p,a}\leq \tau\}|/|P|. \]
The performance profile of an algorithm $a\in A$ corresponds to cumulative distribution function~$\rho_a$.

\perfset{runtime-random-sparse-data-ext.tex}

    \begin{figure}[t!]
        \centering
        \begin{subfigure}[t]{0.45\textwidth}
        \centering
        \begin{tikzpicture}
            \begin{axis}[height=6cm,
            extra x ticks={1},
grid=both, 
no marks, xlabel={Normalized time $\tau$}, ylabel={Proportion of  instances $\rho$}, xmin=1, xmax=20,
            label style={font=\footnotesize},
            tick label style={font=\footnotesize},    title style={font=\small},
            legend style={font=\footnotesize, at={(0.95,0.3)},anchor=east}]
                \addprofiles{4}{20}
                \legend{\textsc{bc+m}, \textsc{bc+g}, \textsc{bc+i}, \textsc{bc}}
            \end{axis}
        \end{tikzpicture}
        \end{subfigure}
\hfill
        \begin{subfigure}[t]{0.45\textwidth}
        \centering
        \begin{tikzpicture}            
            \begin{axis}[height=6cm,
            extra x ticks={1},
grid=both,no marks,
            xlabel={Normalized time $\tau$}, 
xmin=1, xmax=100,
            label style={font=\footnotesize},
            tick label style={font=\footnotesize},    title style={font=\footnotesize},
            legend style={font=\footnotesize, at={(0.95,0.3)},anchor=east}]
                \addprofiles{4}{100}
                \legend{\textsc{bc+m}, \textsc{bc+g}, \textsc{bc+i}, \textsc{bc}}
            \end{axis}
\end{tikzpicture}
        \end{subfigure}\caption{Performance profiles for the random graph instances.}
\label{fig:perf-profile-random}
    \end{figure} \begin{filecontents}{gap-multi.csv}
1.53,2.06,4.39,1.24,0.2,0.43,4.69,9.04,8.41,3.17,5.27,0.79,0.6,0.3,3.3,0.75
\end{filecontents}

\begin{figure*}[t!]
\centering
\begin{subfigure}[t]{1\textwidth}
\centering
\begin{tikzpicture}
	\pgfplotstableread[col sep=comma]{gap-multi.csv}\csvdataA
\pgfplotstabletranspose\datatransposedA{\csvdataA}

	\pgfplotstableread[col sep=comma]{gap-gencon.csv}\csvdataB
\pgfplotstabletranspose\datatransposedB{\csvdataB}

    \pgfplotstableread[col sep=comma]{gap-ind.csv}\csvdataC
\pgfplotstabletranspose\datatransposedC{\csvdataC}

    \pgfplotstableread[col sep=comma]{gap-con.csv}\csvdataD
\pgfplotstabletranspose\datatransposedD{\csvdataD}
    
	\begin{axis}[
        y=0.2cm,
        x=1.2cm,
        xminorgrids=true,
		boxplot/draw direction = y,
		x axis line style = {opacity=0},
		axis x line* = bottom,
		axis y line* = left,
enlarge y limits={abs=0.4cm},
		ymajorgrids,
		xtick = {1, 2, 3, 4},
		xticklabel style = {align=center, font=\footnotesize, rotate=0},
		xticklabels = {\textsc{bc+m}, \textsc{bc+g}, \textsc{bc+i}, \textsc{bc}},
		xtick style = {draw=none}, ylabel = {Gap \%},
boxplot/whisker range={100},
label style={font=\footnotesize},
title style={font=\footnotesize},
        legend style={font=\large}, 
        x tick label style={font=\small},
        y tick label style={font=\footnotesize},
	]
        \addplot+[boxplot, fill, draw=black] table[y index=1] {\datatransposedA};
		\addplot+[boxplot, fill, draw=black] table[y index=1] {\datatransposedB};
        \addplot+[boxplot, fill, draw=black] table[y index=1] {\datatransposedC};
        \addplot+[boxplot, fill, draw=black] table[y index=1] {\datatransposedD};
\end{axis}
\end{tikzpicture}
\end{subfigure}
\caption{Gaps for the random graph instances.}\label{fig:gaps}
\end{figure*} 
Out of the first set of 250 instances formed by (general) random graphs, 
the total number of instances not solved by \textsc{bc}, \textsc{bc+i}, \textsc{bc+g}, and \textsc{bc+m} is 23, 25, 25, and 18, respectively.
The performance profiles depicted in Figure~\ref{fig:perf-profile-random} show that \textsc{bc+m}  is the fastest solving method for over $75\%$ of the instances, and clearly outperforms \textsc{bc+g}, \textsc{bc+i}, and \textsc{bc}.
The second best algorithm is \textsc{bc+g}, which is the fastest for over $70\%$ of the instances.
Figure~\ref{fig:gaps} depicts the final gaps for the instances not solved to optimality by means of box plots; we observe that \textsc{bc+m} produces smaller gaps than \textsc{bc+g}, \textsc{bc+i}, and \textsc{bc}.
The average gaps across all 250 instances computed by  \textsc{bc}, \textsc{bc+i}, \textsc{bc+g}, and \textsc{bc+m} are 0.57\%, 0.55\%, 0.34\%, and 0.18\%, respectively.

\perfset{runtime-bipartite-data.tex}

    \begin{figure}[t!]
        \centering
        \begin{subfigure}[t]{0.45\textwidth}
        \centering
        \begin{tikzpicture}
            \begin{axis}[height=6cm,
            extra x ticks={1},
grid=both, 
no marks, xlabel={Normalized time $\tau$}, ylabel={Proportion of  instances $\rho$}, xmin=1, xmax=20,
            label style={font=\footnotesize},
            tick label style={font=\footnotesize},    title style={font=\small},
            legend style={font=\footnotesize, at={(0.95,0.3)},anchor=east}]
                \addprofiles{4}{20}
                \legend{\textsc{bc+m}, \textsc{bc+g}, \textsc{bc+i}, \textsc{bc}}
            \end{axis}
        \end{tikzpicture}
        \end{subfigure}
\hfill
        \begin{subfigure}[t]{0.45\textwidth}
        \centering
        \begin{tikzpicture}            
            \begin{axis}[height=6cm,
            extra x ticks={1},
grid=both,no marks,
            xlabel={Normalized time $\tau$}, 
xmin=1, xmax=100,
            label style={font=\footnotesize},
            tick label style={font=\footnotesize},    title style={font=\small},
            legend style={font=\footnotesize, at={(0.95,0.3)},anchor=east}]
                \addprofiles{4}{100}
                \legend{\textsc{bc+m}, \textsc{bc+g}, \textsc{bc+i}, \textsc{bc}}
            \end{axis}
\end{tikzpicture}
        \end{subfigure}\caption{Performance profiles for the bipartite instances with $p \in \{0.05, 0.10, 0.15, 0.20, 0.25\}$.}
        \label{fig:perf-profile-bipartite-sparse}
    \end{figure} 

\perfset{runtime-bipartite-dense-data.tex}

    \begin{figure}[t!]
        \centering
        \begin{subfigure}[t]{0.45\textwidth}
        \centering
        \begin{tikzpicture}
            \begin{axis}[height=6cm,
            extra x ticks={1},
grid=both, 
no marks, xlabel={Normalized time $\tau$}, ylabel={Proportion of  instances $\rho$}, xmin=1, xmax=20,
            label style={font=\footnotesize},
            tick label style={font=\footnotesize},    title style={font=\small},
            legend style={font=\footnotesize, at={(0.95,0.5)},anchor=east}]
                \addprofiles{4}{20}
                \legend{\textsc{bc+m}, \textsc{bc+g}, \textsc{bc+i}, \textsc{bc}}
            \end{axis}
        \end{tikzpicture}
        \end{subfigure}
\hfill
        \begin{subfigure}[t]{0.45\textwidth}
        \centering
        \begin{tikzpicture}            
            \begin{axis}[height=6cm,
            extra x ticks={1},
grid=both,no marks,
            xlabel={Normalized time $\tau$}, 
xmin=1, xmax=100,
            label style={font=\footnotesize},
            tick label style={font=\footnotesize},    title style={font=\small},
            legend style={font=\footnotesize, at={(0.95,0.5)},anchor=east}]
                \addprofiles{4}{100}
                \legend{\textsc{bc+m}, \textsc{bc+g}, \textsc{bc+i}, \textsc{bc}}
            \end{axis}
\end{tikzpicture}
        \end{subfigure}\caption{Performance profiles for the bipartite instances with $p \in \{0.30, 0.35, 0.40, 0.45, 0.50\} $.}
        \label{fig:perf-profile-bipartite-dense}
    \end{figure} \begin{filecontents}{gap-bip-multi.csv}
5.67,18.01,5.1,4,12.06,2.78,2.8,1.46,4.54,1.85,5.16,3.61,3.8,2.06,0.07,0.36,1.77,1,0.66,0.57,2.75,2.1,1.4,0.36,0.56,0.86,0.23,0.29,1.71,0.88,0.92,0.38,2.41,1.18,1.45,0.45,0.25,4.16,1.86,0.35,1.63,3.18,2.45,0.85,0.81,2.48,4.23,0.84,0.25,0.49,0.81,0.14,0.07,0.08,0.32,0.78,0.1,0.08,0.62
\end{filecontents}

\begin{figure*}[t!]
\centering
\begin{subfigure}[t]{1\textwidth}
\centering
\begin{tikzpicture}
	\pgfplotstableread[col sep=comma]{gap-bip-multi.csv}\csvdataA
\pgfplotstabletranspose\datatransposedA{\csvdataA}

	\pgfplotstableread[col sep=comma]{gap-bip-gencon.csv}\csvdataB
\pgfplotstabletranspose\datatransposedB{\csvdataB}

    \pgfplotstableread[col sep=comma]{gap-bip-ind.csv}\csvdataC
\pgfplotstabletranspose\datatransposedC{\csvdataC}

    \pgfplotstableread[col sep=comma]{gap-bip-con.csv}\csvdataD
\pgfplotstabletranspose\datatransposedD{\csvdataD}
    
	\begin{axis}[
y=0.1cm,
x=1.2cm,
        xminorgrids=true,
		boxplot/draw direction = y,
		x axis line style = {opacity=0},
		axis x line* = bottom,
		axis y line* = left,
enlarge y limits={abs=0.4cm},
		ymajorgrids,
		xtick = {1, 2, 3, 4},
		xticklabel style = {align=center, font=\footnotesize, rotate=0},
		xticklabels = {\textsc{bc+m}, \textsc{bc+g}, \textsc{bc+i}, \textsc{bc}},
		xtick style = {draw=none}, ylabel = {Gap \%},
boxplot/whisker range={100},
label style={font=\small},
        tick label style={font=\small},    title style={font=\small},
        legend style={font=\footnotesize},
        x tick label style={font=\small},
        y tick label style={font=\footnotesize},
	]
        \addplot+[boxplot, fill, draw=black] table[y index=1] {\datatransposedA};
		\addplot+[boxplot, fill, draw=black] table[y index=1] {\datatransposedB};
        \addplot+[boxplot, fill, draw=black] table[y index=1] {\datatransposedC};
        \addplot+[boxplot, fill, draw=black] table[y index=1] {\datatransposedD};
\end{axis}
\end{tikzpicture}
\end{subfigure}
\caption{Gaps for the bipartite instances.}\label{fig:boxplot-bip}
\end{figure*}

To better assess the impact of the proposed cuts, we also use the set of 250 bipartite instances, which, as we shall see, are computationally more challenging than the previous random graph instances.
The performance profiles depicted in Figure~\ref{fig:perf-profile-bipartite-sparse} and~\ref{fig:perf-profile-bipartite-dense}  show that
\textsc{bc+m} is the fastest solving method for most instances, clearly outperforming all other algorithms.
The number of bipartite instances not solved by \textsc{bc}, \textsc{bc+i}, \textsc{bc+g}, and \textsc{bc+m} is 195, 85, 79, and 59, respectively; most of these instances have $p \leq 0.25$.
From Figure~\ref{fig:boxplot-bip} we see that the positive final gaps  produced by \textsc{bc+m} and \textsc{bc+g} are considerably smaller than the ones yielded by \textsc{bc+i} and \textsc{bc}.
The average gaps across all 250 instances for the algorithms \textsc{bc}, \textsc{bc+i}, \textsc{bc+g}, and \textsc{bc+m} are 2.97\%, 1.65\%, 0.53\%, and 0.49\%, respectively.

The foregoing computational results show that both general connectivity and multiway inequalities can lead to significant speedups in running times and tighter gaps, even when separated using simple heuristics like those described in Section~\ref{sec:exp:separation}.

 \section{Further research}
\label{sec:conclusion}

The ILP formulation of $\poly(G,k)$ given by~\eqref{ineq:cover},~\eqref{ineq:cut}, and $x \in \{0,1\}^{nk}$ is natural and it can be adapted to model additional specific constraints of many problems involving connected partitions.
Additionally, one may extend it to model connected $k$-subpartitions where a class $c \in [k]$ has to be $\kappa_c$-connected (with $\kappa_c \in \Z_\geq$) by simply imposing 
\begin{align*}
    \kappa_c(x_{u,c} +  x_{v,c} -1 )  \leq  \sum_{z \in Z} x_{z,c}  & \quad  \forall  \{u,v\} \notin E(G), Z \in \Gamma(u,v). 
\end{align*}
Similar inequalities were used to model a problem in wireless sensor networks which is related to connected dominating sets~(see Ahn and Park~\cite{Ahn2015}).
One may also have different costs/gains when assigning a vertex to a class as in the case of the \textsc{Convex Recoloring} problem.
In general, one may easily ensure different properties to the classes in the  (sub)partition using variables indexed by the vertices and classes.

If $k$ is large and the classes in the subpartition have identical properties, the studied formulation suffers from symmetry in the sense that any permutation of the classes of a solution lead to an equivalent solution. 
To avoid such symmetries, one could alternatively model~$\poly(G,k)$ using binary variables $\tilde x_{uv} \in \{0,1\}$ for every $u,v \in V(G)$ to identify whether $u$ and~$v$ belong to the same class in a connected $k$-subpartition. 
A future research direction is to find valid inequalities to strengthen such a formulation and investigate the properties of instances for which one formulation outperforms the other.

Another natural direction for further research is to investigate new classes of valid inequalities since there exist facets of $\poly(G,k)$ different from the ones we studied.
Using PORTA~\cite{PORTA},  we found that the following inequalities are facet-defining for the graph depicted in Figure~\ref{fig:graph-porta} with classes $a$ and $b$:
\begin{align*}
     x_{v_1,a}+ x_{v_2,a}+x_{v_3,a}- x_{v_4,a}    +2 x_{v_1,b}+2x_{v_2,b}+x_{v_3,b}-2x_{v_4,b}-2x_{v_5,b} & \leq 3,\\
    2x_{v_1,a}+2x_{v_2,a}+x_{v_3,a}-2x_{v_4,a}-2x_{v_5,a}+ x_{v_1,b}+ x_{v_2,b}+x_{v_3,b}- x_{v_4,b}     & \leq 3.
\end{align*}
Finally, to implement more efficient solving methods, it would be necessary to better understand the computational complexity of the separation problems associated with the proposed inequalities, and devise good heuristics to tackle such problems.

\begin{figure}[t!]
    \centering
    \begin{tikzpicture}[
auto,
            node distance = 1.5cm, semithick, scale=0.8, 
            inner sep=1pt
]
\tikzstyle{every state}=[
            draw = black,
            thick,
minimum size = 4mm,
            scale=0.9
        ]

        \node[state] (v4) {$v_4$};
        \node[state] (v1) [above right of=v4] {$v_1$};
        \node[state] (v2) [below right of=v4] {$v_2$};
        \node[state] (v3) [below left of = v4] {$v_3$};
\node[state] (v5) [below right of=v1] {$v_5$};

        \path (v4) edge node{} (v1);
        \path (v4) edge node{} (v2);
        \path (v4) edge node{} (v3);
        \path (v5) edge node{} (v1);
        \path (v5) edge node{} (v2);
        
    \end{tikzpicture}
\caption{A graph inducing a polytope (with $k=2$) that has facets not yet 
    identified.\label{fig:graph-porta}}
\end{figure}

\section*{Statements and Declarations}
\textbf{Competing interests.} The authors declare that they have no competing interest. 
\hfill \\

\noindent\textbf{Availability of data and materials.} 
The dataset used in this study was synthetically generated. While the generated data is not publicly available at this time, the methodology and generation process are described in detail in the manuscript to ensure reproducibility.

\bibliographystyle{abbrvnat}      
\bibliography{biblio}

@article{dolan2002benchmarking,
  title={Benchmarking optimization software with performance profiles},
  author={Dolan, Elizabeth D and Mor{\'e}, Jorge J},
  journal={Mathematical programming},
  volume={91},
  pages={201--213},
  year={2002},
  publisher={Springer}
}

@article{Ahuja1992,
	Author = {Ahuja, Ravindra K. and Goldberg, Andrew V. and Orlin, James B. and Tarjan, Robert E.},
	Da = {1992/01/01},
	Doi = {10.1007/BF01585705},
	Id = {Ahuja1992},
	Isbn = {1436-4646},
	Journal = {Mathematical Programming},
	Number = {1},
	Pages = {243--266},
	Title = {Finding minimum-cost flows by double scaling},
	Ty = {JOUR},
	Url = {https://doi.org/10.1007/BF01585705},
	Volume = {53},
	Year = {1992}}

@article{Ahn2015,
	Author = {Ahn, Namsu and Park, Sungsoo},
	Da = {2015/04/01},
	Doi = {10.1007/s11276-014-0819-6},
	Id = {Ahn2015},
	Isbn = {1572-8196},
	Journal = {Wireless Networks},
	Number = {3},
	Pages = {783--792},
	Title = {An optimization algorithm for the minimum $k$-connected $m$-dominating set problem in wireless sensor networks},
	Ty = {JOUR},
	Url = {https://doi.org/10.1007/s11276-014-0819-6},
	Volume = {21},
	Year = {2015}}

@article{Hojny21,
	Author = {Hojny, Christopher and Joormann, Imke and L{\"u}then, Hendrik and Schmidt, Martin},
	Da = {2021/03/01},
	Doi = {10.1007/s12532-020-00186-3},
	Id = {Hojny2021},
	Isbn = {1867-2957},
	Journal = {Mathematical Programming Computation},
	Number = {1},
	Pages = {75--132},
	Title = {Mixed-integer programming techniques for the connected max-k-cut problem},
	Ty = {JOUR},
	Url = {https://doi.org/10.1007/s12532-020-00186-3},
	Volume = {13},
	Year = {2021}
}

@article{Lovasz77,
  title={A homology theory for spanning trees of a graph},
  author={L{\'a}szl{\'o} Lov{\'a}sz},
  journal={Acta Mathematica Academiae Scientiarum Hungarica},
  year={1977},
  volume={30},
  pages={241-251}
}

@inproceedings{Gyori78,
  title={On division of graph to connected subgraphs},
  author={Gy{\"o}ri, E.},
  booktitle={In: Combinatoris (Proc. Fifth Hungarian Colloq., Koszthely, 1976), vol. I,
  Colloq. Math. Soc. J{\'a}nos Bolyai},
  volume={18},
  pages={485-494},
  year={1978},
  publisher={North-Holland, Amsterdam, New York}}

@article{Fischetti2017,
	Author = {Fischetti, Matteo and Leitner, Markus and Ljubi{\'c}, Ivana and Luipersbeck, Martin and Monaci, Michele and Resch, Max and Salvagnin, Domenico and Sinnl, Markus},
	Da = {2017/06/01},
	Doi = {10.1007/s12532-016-0111-0},
	Id = {Fischetti2017},
	Isbn = {1867-2957},
	Journal = {Mathematical Programming Computation},
	Number = {2},
	Pages = {203--229},
	Title = {Thinning out Steiner trees: a node-based model for uniform edge costs},
	Ty = {JOUR},
	Url = {https://doi.org/10.1007/s12532-016-0111-0},
	Volume = {9},
	Year = {2017}}

@misc{PORTA,
  title = {{PORTA} - {PO}lyhedron {R}epresentation {T}ransformation {A}lgorithm},
    key = {PORTA},
  howpublished = {\url{https://porta.zib.de}},
  note = {Accessed: 2023-10-17}
}

@book{ahuja1993network,
  title={Network Flows: Theory, Algorithms, and Applications},
  author={Ahuja, R.K. and Magnanti, T.L. and Orlin, J.B.},
  isbn={9780136175490},
  lccn={lc92026702},
  url={https://books.google.be/books?id=WnZRAAAAMAAJ},
  year={1993},
  publisher={Prentice Hall}
}

@article{CAMPELO202254,
title = {Strong inequalities and a branch-and-price algorithm for the convex recoloring problem},
journal = {European Journal of Operational Research},
volume = {303},
number = {1},
pages = {54-65},
year = {2022},
issn = {0377-2217},
doi = {https://doi.org/10.1016/j.ejor.2022.02.013},
url = {https://www.sciencedirect.com/science/article/pii/S037722172200114X},
author = {Manoel Campêlo and Alexandre S. Freire and Phablo F. S. Moura and Joel C. Soares}
}

@article{MOURA2020252,
title = {Strong intractability results for generalized convex recoloring problems},
journal = {Discrete Applied Mathematics},
volume = {281},
pages = {252-260},
year = {2020},
issn = {0166-218X},
doi = {https://doi.org/10.1016/j.dam.2019.08.002},
url = {https://www.sciencedirect.com/science/article/pii/S0166218X19303506},
author = {Phablo F. S. Moura and Yoshiko Wakabayashi}
}

@article{Carvajal13,
author = {Carvajal, Rodolfo and Constantino, Miguel and Goycoolea, Marcos and Vielma, Juan Pablo and Weintraub, Andr\'{e}s},
title = {Imposing Connectivity Constraints in Forest Planning Models},
journal = {Operations Research},
volume = {61},
number = {4},
pages = {824-836},
year = {2013},
doi = {10.1287/opre.2013.1183},
URL = {https://doi.org/10.1287/opre.2013.1183},
eprint = {https://doi.org/10.1287/opre.2013.1183}
}

@article{hochbaum1994node,
  title={Node-optimal connected $k$-subgraphs},
  author={Hochbaum, Dorit S and Pathria, Anu},
  journal={manuscript, UC Berkeley},
  year={1994}
}

@Incollection{Alvares-Miranda2013,
author="{\'A}lvarez-Miranda, Eduardo and Ljubi{\'{c}}, Ivana and Mutzel, Petra",
editor="J{\"u}nger, Michael and Reinelt, Gerhard",
title="The Maximum Weight Connected Subgraph Problem",
bookTitle="Facets of Combinatorial Optimization: Festschrift for Martin Gr{\"o}tschel",
year="2013",
publisher="Springer Berlin Heidelberg",
address="Berlin, Heidelberg",
pages="245--270"
}

@article{CALINESCU2003333,
title = {Multicuts in unweighted graphs and digraphs with bounded degree and bounded tree-width},
journal = {Journal of Algorithms},
volume = {48},
number = {2},
pages = {333-359},
year = {2003},
issn = {0196-6774},
doi = {https://doi.org/10.1016/S0196-6774(03)00073-7},
url = {https://www.sciencedirect.com/science/article/pii/S0196677403000737},
author = {Gruia Călinescu and Cristina G. Fernandes and Bruce Reed}
}

@article{ArrMarPenRic21,
  Author = {Arredondo, Ver{\'o}nica and Mart{\'\i}nez-Panero, Miguel and Pe{\~n}a, Teresa and Ricca, Federica},
  Da = {2021/08/17},
  Journal = {Annals of Operations Research},
  Title = {Mathematical political districting taking care of minority groups},
  Year = {2021},
  volume = "305",
  pages = "375--402",
  Url = {https://doi.org/10.1007/s10479-021-04227-5}
}

@article{LucPerSim93,
title = "Most uniform path partitioning and its use in image processing",
journal = "Discrete Applied Mathematics",
volume = "42",
number = "2",
pages = "227--256",
year = "1993",
author = "Mario Lucertini and Yehoshua Perl and Bruno Simeone"
}

@article{MarSimNal97,
title = "Clustering on trees",
journal = "Computational Statistics \& Data Analysis",
volume = "24",
number = "2",
pages = "217--234",
year = "1997",
author = "Maurizio Maravalle and Bruno Simeone and Rosella Naldini",
}

@article{MorSni08,
    title = "Convex recolorings of strings and trees: Definitions, hardness results and algorithms ",
    journal = "Journal of Computer and System Sciences ",
    volume = "74",
    number = "5",
    pages = "850--869",
    year = "2008",
    issn = "0022-0000",
    doi = "10.1016/j.jcss.2007.10.003",
    url = "http://www.sciencedirect.com/science/article/pii/S0022000007001468",
    author = "Shlomo Moran and Sagi Snir"
}

@inproceedings{ORLIN2013,
author = {Orlin, James B.},
title = {Max Flows in \rm{$O(nm)$} Time, or Better},
year = {2013},
isbn = {9781450320290},
publisher = {Association for Computing Machinery},
address = {New York, NY, USA},
url = {https://doi.org/10.1145/2488608.2488705},
doi = {10.1145/2488608.2488705},
booktitle = {Proceedings of the Forty-Fifth Annual ACM Symposium on Theory of Computing},
pages = {765–774},
numpages = {10},
location = {Palo Alto, California, USA},
series = {STOC '13}
}

@book{korte2012greedoids,
  title={Greedoids},
  series={Algorithms and Combinatorics},
  author={Korte, Bernhard and Lov{\'a}sz, L{\'a}szl{\'o} and Schrader, Rainer},
  volume={4},
  year={1991},
  publisher={Springer Berlin, Heidelberg},
  doi={https://doi.org/10.1007/978-3-642-58191-5}
}

@article{CAMPELO2013,
title = {Polyhedral studies on the convex recoloring problem},
journal = {Electronic Notes in Discrete Mathematics},
volume = {44},
pages = {233-238},
year = {2013},
issn = {1571-0653},
doi = {https://doi.org/10.1016/j.endm.2013.10.036},
url = {https://www.sciencedirect.com/science/article/pii/S1571065313002527},
author = {Manoel Campêlo and Karla R. Lima and Phablo F. S. Moura and Yoshiko Wakabayashi}}

@article{CAMPELO2016,
	Author = {Camp{\^e}lo, Manoel and Freire, Alexandre S. and Lima, Karla R. and Moura, Phablo F. S. and Wakabayashi, Yoshiko},
	Da = {2016/03/01},
	Doi = {10.1007/s10107-015-0880-7},
	Id = {Camp{\^e}lo2016},
	Isbn = {1436-4646},
	Journal = {Mathematical Programming},
	Number = {1},
	Pages = {303--330},
	Title = {The convex recoloring problem: Polyhedra, facets and computational experiments},
	Ty = {JOUR},
	Url = {https://doi.org/10.1007/s10107-015-0880-7},
	Volume = {156},
	Year = {2016}}

@article{MIYAZAWA2021826,
title = {Partitioning a graph into balanced connected classes: Formulations, separation and experiments},
journal = {European Journal of Operational Research},
volume = {293},
number = {3},
pages = {826-836},
year = {2021},
issn = {0377-2217},
doi = {https://doi.org/10.1016/j.ejor.2020.12.059},
url = {https://www.sciencedirect.com/science/article/pii/S0377221720311218},
author = {Flávio K. Miyazawa and Phablo F. S. Moura and Matheus J. Ota and Yoshiko Wakabayashi}
}

@article{wang2017imposing,
  title={On imposing connectivity constraints in integer programs},
  author={Wang, Yiming and Buchanan, Austin and Butenko, Sergiy},
  journal={Mathematical Programming},
  volume={166},
  number={1},
  pages={241--271},
  year={2017},
  publisher={Springer}
}

@article{GARG200449,
title = {Multiway cuts in node weighted graphs},
journal = {Journal of Algorithms},
volume = {50},
number = {1},
pages = {49-61},
year = {2004},
issn = {0196-6774},
doi = {https://doi.org/10.1016/S0196-6774(03)00111-1},
url = {https://www.sciencedirect.com/science/article/pii/S0196677403001111},
author = {Naveen Garg and Vijay V. Vazirani and Mihalis Yannakakis}
}

@article{balas2005vertex,
  title={The vertex separator problem: a polyhedral investigation},
  author={Balas, Egon and Souza, Cid C de},
  journal={Mathematical Programming},
  volume={103},
  number={3},
  pages={583--608},
  year={2005},
  publisher={Springer}
}

@article{CHLEBIK2008292,
title = {Crown reductions for the Minimum Weighted Vertex Cover problem},
journal = {Discrete Applied Mathematics},
volume = {156},
number = {3},
pages = {292-312},
year = {2008},
issn = {0166-218X},
doi = {https://doi.org/10.1016/j.dam.2007.03.026},
url = {https://www.sciencedirect.com/science/article/pii/S0166218X07001308},
author = {Miroslav Chlebík and Janka Chlebíková}
}

@article{chopra2017extended,
  title={An extended formulation of the convex recoloring problem on a tree},
  author={Chopra, Sunil and Filipecki, Bartosz and Lee, Kangbok and Ryu, Minseok and Shim, Sangho and Van Vyve, Mathieu},
  journal={Mathematical Programming},
  volume={165},
  pages={529--548},
  year={2017},
  publisher={Springer}
}

\end{document}